\documentclass[12pt]{article}
\usepackage[francais,english]{babel}
\usepackage{amsmath}
\usepackage{amsfonts}
\usepackage{amssymb}
\usepackage{amsthm}
\usepackage{mhequ}
\usepackage{amscd}
\usepackage{fullpage}
\usepackage[matrix,arrow,ps]{xy}

\usepackage{hyperref}
\usepackage[sort=def]{glossaries}   %type makeglossaries filename in terminal
\usepackage[pdftex]{graphicx,color} 

\newcommand{\Z}{\mathbb{Z}}
\newcommand{\R}{\mathbb{R}}

\newcommand{\C}{\mathbb{C}}

\newcommand{\E}{\mathbb{E}}

\newcommand{\mc}{\mathcal}

\newcommand{\eps}{\varepsilon}
\newcommand{\ind}{{\bf 1}}

\DeclareMathOperator{\grad}{grad}

\newcommand{\TT}{\Lambda}

\newcommand{\Lap}{\Delta\!}

\newcommand{\m}{\mathbf{m}}

\newcommand{\area}{\omega}
\newcommand{\MeaSp}{\mathcal{X}}

\newcommand{\EE}{\mc{E}_{\MeaSp}}

\def\Wick#1{{:}{#1}{:}}
\def\M{M_\phi}

\def\TestSp{\mathcal C}
\def\TestSe{\tilde{\mathcal C}}
\def\CX{{\mathcal C}_{\mathcal X}}

\newcommand{\vertiii}[1]{{\left\vert\kern-0.25ex\left\vert\kern-0.25ex\left\vert #1 
    \right\vert\kern-0.25ex\right\vert\kern-0.25ex\right\vert}}

\newtheorem{thm}{Theorem}[section]

\newtheorem{Thm}[thm]{Theorem}
\newtheorem{Def}[thm]{Definition}
\newtheorem{Prop}[thm]{Proposition}
\newtheorem{Lem}[thm]{Lemma}
\newtheorem{Cor}[thm]{Corollary}

\newtheorem{Rem}[thm]{Remark}
\def\XXint#1#2#3{{\setbox0=\hbox{$#1{#2#3}{\int}$ }
\vcenter{\hbox{$#2#3$ }}\kern-.6\wd0}}

\def\HuntA{\mathbf A}

\begin{document}

\title{Stochastic Ricci Flow on Compact Surfaces}
\author{Julien Dub\'edat\footnote{Julien Dub\'edat, Columbia University, Email: dubedat@math.columbia.edu}
  $\;$ and Hao Shen\footnote{Hao Shen, University of Wisconsin-Madison, Email: pkushenhao@gmail.com}
}

\maketitle

\abstract{In this paper we introduce  
the stochastic Ricci flow (SRF)  in two spatial dimensions. The flow is symmetric with respect to a measure induced by Liouville Conformal Field Theory.
Using the theory of Dirichlet forms,
we construct a weak solution  to the associated equation of 
 the area measure on a flat torus, in the full  ``$L^1$ regime'' $\sigma< \sigma_{L^1}=2 \sqrt\pi$ where $\sigma$ is the noise strength.
We also describe the  main necessary modifications needed for the SRF on general compact surfaces,
and list some open questions. 
 }

\tableofcontents

\section{Introduction}
\label{sec:Intro}

The Ricci flow, introduced by Hamilton \cite{MR954419}, is an intrinsic evolution of a Riemannian metric $g=g(t)$ on a fixed smooth manifold:
\[
\partial_t g=-2R_g
\]
 where $R_g$ is the Ricci curvature of $g$.
Assuming that the manifold is a closed (compact oriented) Riemann surface, we will be interested in a normalized version  of the flow:
 \begin{equ}[e:lambda-Ricci]
 \partial_t g=-2R_g-2\lambda g \;.
 \end{equ}
The real number $\lambda$ plays the role of a normalization constant: when  $\lambda$ is  chosen to be the minus average Gauss curvature, the flow preserves total area, and converges to the constant curvature metric as proved by \cite{MR954419,MR1094458,OPS}. 

Up to normalization, the Ricci flow \eqref{e:lambda-Ricci} is the only intrinsic parabolic (deterministic) evolution of the metric. 
Indeed, the coefficients $-2$ do not play any essential role; 
%a simple time change $t'=at$ and change of unknown $g'(t)=g(t)e^{bt}$, one gets a dynamic of the form 
%$\partial_t g=-\sigma^2R_g-b g$ for arbitrary $\sigma^2>0$ and $b\in\R$. 
a simple transform $\tilde{g}(t)=g(at)e^{bt}$ yields the same equation with generic coefficients.
Moreover, classical results in Riemannian geometry (e.g. \cite{Epstein}) show that, under mild assumptions, the only ``natural" tensors (i.e. that do not depend on coordinate or other choices) are those generated by the metric and curvature tensor by taking linear combinations, covariant derivatives, tensor products and contractions; in two dimensions the only intrinsic terms are the metric and Ricci curvature. In the present article, we are concerned with constructing intrinsic {\em stochastic} evolutions of the metric.

Recall the following facts. Let $g_0$ be a metric on a closed Riemann surface.
We consider 
 metrics obtained by Weyl scaling $g=e^{2\phi}g_0$,
where the function $\phi$ is  the conformal factor. 
The Ricci curvature tensor $R_g$, the Gauss curvature $K_g$, the Laplacian\footnote{In this article the Laplacian is just the trace of the Hessian on functions, {\it without} a negative sign.}
 operator $\Lap_g$ and the 
area form $\area_g$ for $g$ then satisfy
 \begin{equ} [e:useful-facts]
 R_g=  R_0 - \Lap_0\phi \, g_0 \;,
 \qquad
 K_g=e^{-2\phi}(K_0-\Lap_0\phi) \;,
 \qquad
 \Lap_g=e^{-2\phi}\Lap_0\;,
 \qquad
 \area_g = e^{2\phi} \area_0 \;,
 \end{equ}
 where $R_0$, $\Lap_0$ and $\area_0$ are respectively the Ricci curvature, the Laplacian  and area form for $g_0$. 
 Since the dimension is two, the scalar curvature (the trace of the Ricci tensor w.r.t. the metric) equals twice the Gauss curvature, and one always has $R_g=K_g \,g$.
An important  property  of the 2d Ricci flow is that 
the metric  evolves within a conformal class, namely if the initial condition has the above form, then $g(t)=e^{2\phi(t)}g_0$ for all $t>0$,
so that the equation \eqref{e:lambda-Ricci} can be written in terms of the conformal factor in the following equivalent forms:
\begin{equs}[e:lambda-Ricci-phi]
\partial_t\phi = -K_g-\lambda
&=\Delta_g \phi - e^{-2\phi}K_0-\lambda \\
&=e^{-2\phi}\Delta_0\phi - e^{-2\phi}K_0   -\lambda 
\;.
\end{equs}

To motivate our stochastic version of the two-dimensional Ricci flow, as well as its 
 connection to the Liouville Conformal Field Theory (LCFT),
we recall that String Theory is concerned with surfaces with varying metrics $g$, and a central quantity is the (``$\zeta$-regularized'') determinant of Laplacian formally denoted by 
$\det' \Delta_g $. This a spectral invariant, i.e. can be computed from the spectrum of $\Delta_g$ (which is discrete and satisfies Weyl asymptotics). 
Osgood-Phillips-Sarnak \cite{OPS} provided an important
perspective on the uniformization theorem:
in \cite[Theorem~1]{OPS} they proved that
among all metrics in a given conformal class and of given area, the constant curvature metric  (which exists and is unique up to isometry) has maximum determinant $\det' \Delta_g $.
They also showed that \cite[Theorem~2.A]{OPS} when the Euler characteristic is non-positive,
the Ricci flow can be realized as a {\it gradient flow}, as we now explain.

On the Riemann surface $\Sigma$,
the celebrated Polyakov \cite{polyakov1981quantum}
anomaly formula  (see e.g. Section 1 of \cite{OPS}) for the variation of the quantity $\log \det' \Delta_g$ under conformal change $g=e^{2\phi}g_0$ reads:
\begin{equ}[e:Polyakov]
\log{\det}'\Delta_g-\log{\det}'\Delta_{0}
=-\frac 1{12\pi} \int_{\Sigma} |\nabla_{g_0}\phi|^2 \area_0 -\frac 1{6\pi}\int_{\Sigma} K_0\phi \,\area_0
	+\log\frac{V_g}{V_0}
\end{equ}
where  $V_0,V_g$ are areas of $g_0$ and $g$, i.e.
$V_g=\int_{\Sigma} \area_g  =\int_{\Sigma} e^{2\phi} \area_0$;
and $K_0$ is the Gauss curvature of $g_0$.
Define a ``potential'' for $\lambda\in\R$
\[
V(g)=-\log{\det}'\Delta_g+\log V_g+\frac{\lambda}{12\pi} V_g \;.
\]
By \eqref{e:Polyakov} we see that $6\pi(V(g)-V(g_0))$
 is essentially the (classical) {\it Liouville action}  $S(\phi)$ up to a constant 
\begin{equs}[e:6piV-S]
6\pi(V(g) &-V(g_0)) +\frac{\lambda}{2}V_0  \\
&= \;
S(g_0,\phi) 
\;
:= \int_{\Sigma} \Big( \frac12 |\nabla_{g_0} \phi|^2+ K_0\phi + \frac{\lambda}{2}  e^{2 \phi}
	\Big)\,\area_0
\end{equs}
which gives a concrete formula for the action. 
By a simple change  of variables -- see \eqref{e:match-S} in Section~\ref{sec:LiouvilleCFT} below -- 
this is exactly the same as the probabilists' convention of definition of the classical Liouville action,
up to an overall constant $\frac{\pi\gamma^2}{2}$.

The Ricci flow \eqref{e:lambda-Ricci-phi}
turns out to be a gradient flow of $S(g_0,\phi)$
(and thus $6\pi V(\phi)$) given in \eqref{e:6piV-S} with respect to the a formal Riemannian structure on the infinite-dimensional  space of metrics on $\Sigma$ in a fixed conformal class, where the ``tangent space" at $g$ is equipped with the {\it intrinsic}
inner product $L^2(\area_g)$ (rather than the ``flat'' inner product $L^2(\area_0)$!).
Here,  the metric $L^2(\area_g)$ is defined such that for a ``tangent'' vector $\delta\phi$ at a Riemannian metric $g$, 
\begin{equ}[e:tangent]
\|\delta\phi\|^2_{L^2(\area_g)} = \int_{\Sigma} (\delta\phi)^2 \area_g \;.
\end{equ}
This metric is known as the {\em Calabi metric}, see \cite{CalCal} (in the context of K\"ahler metrics, viz. for fixed volume, as well as complex structure).
Consider a perturbation $g+\delta g=e^{2\delta\phi}g$; then
by \eqref{e:6piV-S} and \eqref{e:useful-facts}
\[
\delta S(g_0,\phi)
= \int_{\Sigma} (K_0-\Delta_0 \phi)\delta\phi \, \area_0
	+\lambda\int_{\Sigma} \delta\phi \,\area_g 
=\left\langle   K_g+ \lambda,\delta\phi\right\rangle_{L^2(\area_g)}
\]
so that the gradient flow associated to $S(g_0,\phi)=6\pi V$ is
indeed the Ricci flow \eqref{e:lambda-Ricci-phi}.

In general, associated to, say, a compact Riemannian manifold $({\mc M},{\bf g})$ and a potential ${\bf V}:{\mc M}\rightarrow\R$, one can consider a (deterministic) gradient flow on ${\mc M}$:
$$dX_t=-\grad_{\bf g} V(X_t)$$
The {\em Langevin flow} is a natural stochastic perturbation of the gradient flow associated to the same data $({\mc M},{\bf g},{\bf V})$; its generator reads
\begin{equation}\label{e:langevin}
\sigma^2\Lap_{\bf g}-\grad_{\bf g} {\bf V}(X_t)\cdot\grad_{\bf g}
\end{equation}
and its invariant measure is $\propto e^{-\sigma^{-2}{\bf V}}d{\rm Vol}_{\bf g}$.

Since the Ricci flow is realized as a gradient flow - where ${\mc M}$ is the (infinite-dimensional) space of metrics in a conformal class, ${\bf g}$ is the Calabi metric \eqref{e:tangent} and the potential is given by \eqref{e:6piV-S}, a natural question is to construct the Langevin flow associated to that data. Remark that the (formal !) volume form associated to the Calabi metric is central to the treatment of 2D quantum gravity by David--Distler-Kawai (\cite{David_2DG,DisKaw}, see also Section 2.1 in \cite{Nak_Liouv}).

In view of this, the natural and intrinsic noise
that we would like to add to equation \eqref{e:lambda-Ricci-phi}
should then be a noise which, in the spatial direction, is ``white'' with respect to $L^2(\area_g)$.
In fact if $\zeta_0$ is a {\it spatial} white noise w.r.t. $L^2(\area_0)$, and $\phi$ is a smooth conformal factor, 
%that is, 
then 
  $\zeta_g:=e^{-\phi} \zeta_0$  is white with respect to the metric $L^2(\area_g)$,
namely, by \eqref{e:useful-facts},
\begin{equ}[e:intrin-noise]
\E\big[(\int_{\TT} f\zeta_g \area_g)^2\big]
=\E\big[(\int_{\TT} f \zeta_0 e^{\phi}\area_0)^2\big]
=\int_{\TT} f^2e^{2\phi}\area_0
=\int_{\TT} f^2\area_g \;.
\end{equ}

In this paper  we focus first on a flat two-dimensional torus $\Sigma=\TT=\C/(\Z+\tau\Z)$, $\Im(\tau)>0$. 
We will briefly discuss the case of compact surfaces in Section~\ref{sec:surfaces}, and we will see there 
that while on the torus the stochastic Ricci flow is the classical Ricci flow (in terms of $\phi$) plus a noise that 
is white with respect to the metric $g$,
there are complications when the reference is not flat.

Let $g_0$ be a flat metric on $\TT$ such that $K_0=0$.
 In the sequel we write $\Lap=\Lap_0$.
The  {\it stochastic Ricci flow} (SRF)  we study in this paper is then  {\it formally} given by
\[
\partial_t g=-2R_g-2\lambda g+2\sigma\xi_g g 
\qquad
\mbox{where }\xi_g :=e^{-\phi} \xi_0\;,
\]
or (again formally) in terms of the conformal factor
\begin{equs}[e:SRicci]
\partial_t\phi  &=\Delta_g\phi -\lambda + \sigma \xi_g
\\
&=e^{-2\phi}\Delta\phi -\lambda +  \sigma e^{-\phi} \xi_0 
\;,
\end{equs}
where $\sigma\in\R$ and
 $\xi_0$ is the {\it space-time} white noise w.r.t. the Euclidean metric $g_0$.
 Equation  \eqref{e:SRicci} is a nonlinear version of the stochastic heat equation (SHE)
\begin{equ}[e:SHE]
\partial_t\phi  =\Delta\phi + \sigma \xi_0
\end{equ}
whose invariant measure  is the Gaussian Free Field (GFF) with covariance operator $\frac {\sigma^2}2(-\Lap)^{-1}$. In particular $\phi$ has negative regularity.

\begin{Rem}
Let us point out that, beyond the usual ``flat" SHE, one can consider a SHE with respect to a fixed smooth metric $\hat g=e^{2\hat\phi}g_0$; it reads
\begin{equation}\label{eq:smoothSHE}
\partial_t\phi  =\Delta_{\hat g}\phi + \sigma \xi_{\hat g}=e^{-2\hat\phi}\Delta_0\phi+\sigma e^{-\hat\phi}\xi_0\;.
\end{equation}
For any fixed, smooth $\hat\phi$, the invariant measure is the same GFF. In view of this, it is natural to expect that a solution of \eqref{e:SRicci} has the (negative) regularity of a GFF, and thus all nonlinearities have to be interpreted in a suitable regularized or renormalized way.
\end{Rem}

Note that in \eqref{e:SRicci}, setting $\sigma=0$ recovers the Ricci flow \eqref{e:lambda-Ricci-phi}; setting $\lambda=0$, $\tilde\phi=\sigma^{-1}\phi$, starting from $\phi_0\equiv 0$, and formally taking $\sigma\searrow 0$ (i.e. perturbation around the constant flat solution) gives the stochastic heat equation.

By formally taking time derivative on $\area_g = e^{2\phi} \area_0 $, the evolution for the area form  $\area_g(t)$ is (again formally) given by
\begin{equ}[e:Ricci-A]
\partial_t \area_g(t) = 2\Lap \phi (t) \, \area_0 - 2\lambda\, \area_g(t) + 2 \sigma \xi_g \,\area_g(t) \;.
\end{equ}
Note that the r.h.s. has $\phi$ explicitly showing up. 
Heuristically this is understood as the conformal factor $\phi$ and the area form  $\area_g = e^{2\phi} \area_0 $ mutually determine each other. The precise meaning 
of this mutual determination is discussed below, see the ``inversion'' property of GMC (as proved by \cite{BerestyckiSheffieldSun}) in Section~\ref{Subsec:GMC}.

In this article we will mainly focus on the dynamic \eqref{e:Ricci-A}
and prove Theorem~\ref{thm:main} which provides a meaning of solution to \eqref{e:Ricci-A}.
We provide some background on Gaussian multiplicative chaos and Liouville conformal field theory below before formulating this main theorem. However, let us immediately point out here a key new feature of the stochastic flow - 
for the deterministic flow \eqref{e:lambda-Ricci}, as mentioned above, the total area of the surface can be preserved by choosing the normalization $\lambda$ to be the minus average Gauss curvature which, on the torus and by Gauss-Bonnet, is zero.  This is not the case for  the stochastic flow; in fact the noise $\xi_g$ is a local random perturbation which has no guarantee to preserve the total area as a global quantity.
As we will see below  Theorem~\ref{thm:main}, on the torus the total area $A$
satisfies an SDE 
$dA=2\sigma\sqrt{A}d\beta_t -2\lambda A dt$ where $\beta$ is a 1-dimensional Brownian motion,
and the total area will almost surely vanish in finite time. 
Note that this a.s. finite time vanishing is not in contradiction with 
the fact that the SRF is a Langevin flow for  the Liouville conformal field theory measure
because,  as we will discuss in Section~\ref{sec:LiouvilleCFT}, the Liouville conformal field theory measure being considered here is only a $\sigma$-finite measure, not a probability measure i.e. not normalizable.
We will also describe the evolution and long-time asymptotic 
of the total surface area in Section~\ref{sec:extensions} for general compact surfaces and with insertions of vertex operators, and their relation with Seiberg bound for finiteness of Liouville correlation function.

\subsection{Gaussian multiplicative chaos (GMC)}\label{Subsec:GMC}

The previous discussion suggests that we consider a process of measures $\omega_g(t)=e^{2\phi(t)}\omega_0$, where $\phi$ looks like a GFF. Such {\em Gaussian multiplicative chaos (GMC)} or {\em Liouville measures} have been studied extensively, going back to H\o egh-Krohn \cite{HoeghKrohn}, Kahane \cite{MR829798}; we refer to the survey \cite{MR3274356} and references therein. We list below some basic properties we will need later. We focus on the GFF context relevant to us, although most properties below hold in greater generality (log-correlated fields). See e.g. \cite{Dub_SLEGFF} and references therein for general background on the GFF.

\paragraph{Existence and construction.}

Let $X$ denote a Gaussian Free Field with Dirichlet conditions in a domain $D\subset \C$; it is a centered Gaussian field with covariance given by the Dirichlet Green function:
$$\E(X(z)X(w))=2\pi (-\Delta_D)^{-1}(z,w)=-\log|z-w|+R(z,w)$$
where $R$ is smooth near the diagonal. It is a random distribution (in the sense of Schwartz) and can be realized as a random element of a negative index Sobolev space $H^{-s}_{loc}(D)$, $s>0$, in the abstract Wiener space formalism. 

For $\eps>0$ (and away from the boundary), one can consider the circle average process $X_\eps(z)=\int X(z+\eps e^{i\theta})\frac {d\theta}{2\pi}$. This can be realized as a continuous process, as is easily seen by Kolmogorov's continuity criterion; each $X_\eps$ is measurable w.r.t. $X$. Then one can consider a positive measure on $D$
\begin{equation}\label{eq:epsreg}
M_\eps(X)=\eps^{\frac{\gamma^2}2}\exp(\gamma X_\eps(x))\omega_0(dx)
\end{equation}
where $\gamma$ is a positive parameter and $\omega_0$ is the Lebesgue measure.

Then \cite{MR2819163}, if $0<\gamma<2$, the sequence $(M_\eps)$ converges almost surely in the topology of weak convergence to a positive measure $M_X$, as $\eps$ goes to zero alongs a suitable fixed sequence. The random measure $M_X$ can thus naturally be thought of as a regularized $\Wick{e^{\gamma X}\omega_0}$. It is nonatomic and gives a.s. positive mass to any nonempty open set, and a.s. finite mass to any compact $K\subset D$; it is thus a random element of the space ${\mc M}(D)$ of Radon measures on $D$.

It can be shown that $M_X$ is a.s. supported on $\{z\in D:\lim_{\eps\searrow 0}\frac{X_\eps(z)}{-\log|\eps|}=\gamma\}$ and consequently is a.s. absolutely singular w.r.t. Lebesgue measure. 

If $\gamma\in (0,\sqrt 2)$ and $f$ is a test function, $\int fdM_X$ is a square-integrable; this is the so-called $L^2$ regime and leads to simpler arguments; however the results below hold in the full range $\gamma\in (0,2)$.

Other natural approximation schemes (convolution of $X$ with the heat kernel, or a smooth compactly supported kernel) are possible, and consistent \cite{MR3475456}.

\paragraph{Basic properties.}

A few important properties follow immediately from the construction.
\begin{enumerate}
\item 
{\em Locality.} If $U\subset\subset V\subset\subset D$, then the restriction of $M_X$ to $U$ is measurable w.r.t. the restriction of $X$ to $V$ (this refines the measurability of $X\mapsto M_X$).
\item
{\em Equivariance.} With the previous notation, the mapping $X_{|V}\mapsto (M_X)_{|U}$ does not depend on $D$ and is equivariant w.r.t. Euclidean isometries.
\item {\em Shift.} If $f\in C^1_c(V)$ is fixed, $dM_{f+X}=e^{\gamma f}dM_X$ a.s. (note under these assumptions, the law of $f+X$ and that of $X$ are mutually absolutely continuous). 
\end{enumerate}
Note that shift covariance is central to the approach of \cite{MR3475456}; the condition $f\in C^1_c$ can be relaxed to $f\in  H^1$ which is the Cameron-Martin space. An important additional property is scale (more generally, conformal) covariance. 

\paragraph{Inversion.}
We previously listed properties of the mapping $X\mapsto M_X$, which is defined a.e. on an abstract Wiener space. It will be convenient for our purposes to consider the a.e. defined inverse map $M_X\mapsto X$, constructed by Berestycki-Sheffield-Sun in \cite{BerestyckiSheffieldSun}. More precisely, if $(X,M_X)\in H^{-s}_{loc}(D)\times{\mc M}(D)$ are coupled as above, then $X$ is measurable w.r.t. $M_X$. This shows the existence of an a.e. inverse mapping $M_X\mapsto X$, which is a.e. defined (with respect to the induced measure on the second marginal ${\mc M}(D)$). 

From the explicit construction of \cite{BerestyckiSheffieldSun}, it is clear that this inverse mapping is also local, equivariant, and compatible with shift.  We will denote by $\mathbf M$
the mapping $X\mapsto M_X$, and  $\mathbf M^{-1} $  the mapping $M_X\mapsto X$.

\paragraph{Conventions.}

In our ``geometer's convention'' (coming from the Ricci flow) we would like to consider a Gaussian free field $\phi$ with covariance operator $\frac {\sigma^2}2(-\Lap)^{-1}$; in order to match with the above standard conventions for Liouville measures, let 
\begin{equ}[e:gamma-sigma]
\phi=\frac{\gamma}{2} X
 \qquad
 \mbox{and}
 \qquad
 \sigma = \sqrt\pi\gamma\;,
\end{equ}
 so that 
 \[
 e^{2\phi} = e^{\gamma X} \;.
 \]
 
With our convention, 
 $M_\eps = e^{2\phi_\eps -2 \E (\phi_{\eps}^2)}  \area_0 $ converges to the limit denoted by  $M=M_\phi=\Wick{e^{2\phi} \area_0}$.  The a.e. correspondence $\phi\leftrightarrow M_\phi$ is local and equivariant in the previous sense; the compatibility with shift simply reads
 \begin{equ}[e:shiftcov]
 dM_{f+\phi}=e^{2f}dM_\phi
 \end{equ}
 for a.e. $\phi$, where $f$ is a fixed $H^1$ function. 
 
 Remark that the locality of the correspondence shows that is also holds for a field $\phi$ on a surface whose restriction to small balls is absolutely continuous w.r.t. a GFF there; in particular for $\phi$ a GFF on a flat torus.

 The ``$L^2$ regime'' such that $M(f)$ for a smooth test function $f$ has finite second moment  as well the ``$L^1$ regime'' all the way to which $M$ obtained this way is nontrivial are respectively
 \[
 \sigma<\sigma_{L^2} = \sqrt{2\pi}
 \qquad
  \sigma<\sigma_{L^1} = 2 \sqrt{\pi}
 \]
  (which corresponds to the well-known $\gamma<\gamma_{L^2} =\sqrt{2}$ and $\gamma<\gamma_{L^1} =2$).

\subsection{Liouville conformal field theory}
\label{sec:LiouvilleCFT}

Closely related with GMC is the Liouville CFT measure on the space of fields $X$ over a Riemann surface with a fixed smooth reference metric $g_0$ and volume form $\area_0$, which is given by $Z^{-1}e^{-S(X)}DX$ where 
\[
S(X) = \frac{1}{4\pi} \int \Big( |\nabla X|^2 + 2Q K_0 X+4\pi \mu e^{\gamma X} \Big) \area_0
 \]
where  $Z$ is a normalization factor,  $K_0$ is the Gauss curvature of $g_0$.
The measure (with suitable insertions of vertex operators, see below)
has been rigorously constructed 
 by \cite{MR3465434} (on the sphere, see \cite{kupiainen2016constructive} for a review), and \cite{MR3450564} on the complex tori, and \cite{guillarmou2016polyakov} (genus $\ge 2$);
see also \cite{MR3825895} (on disk) and \cite{MR3843631} (on annulus).
The parameter $\mu > 0$ is the analogue of a ``cosmological constant'' in two dimensional gravity and $Q$ is a real parameter. For the particular value $Q=\frac{2}{\gamma}$ the action functional $S$ is classically conformally covariant. In the quantized theory $Q$ has the renormalized value $Q=\frac{2}{\gamma}+\frac{\gamma}{2}$ such that the random measure $\Wick{e^{\gamma X}}$ is invariant in law under change of reference measure (within a conformal class). Note that if we focus on a torus $\TT$ with flat metric $g_0$ then the necessary correction term $2Q K_0 X$ is hidden.

We remark that when the genus $g\le 1$, 
 the measure $e^{-S(X)}DX$ is not really normalizable 
(i.e. $Z$ is not well-defined) since the integral will diverge as the value of $X$ tends to $-\infty$, 
unless suitable vertex operators (see below) are inserted. 
However, in this paper where we work with torus  $g=1$, we will not consider insertions 
and thus will not normalize the measure. 
Rather, we will view it as a $\sigma$-finite measure, see \eqref{e:def-nu} and Lemma~\ref{Lem:sigmafin} below.

\paragraph{Conventions.}

The action $S(g_0,\phi)$ in \eqref{e:6piV-S} and the stochastic Ricci flow
\eqref{e:SRicci}   depends on two parameters $(\lambda,\sigma)$,
and the  standard conventions for the Liouville CFT action $S(X)$ in the probability literature 
 depends on two parameters $(\mu,\gamma)$. To match the two conventions
 \[
 (\phi,\lambda,\sigma)\qquad \leftrightarrow \qquad  (X,\mu,\gamma)
 \]
  besides the relations \eqref{e:gamma-sigma}, we further set $\lambda = \pi \mu\gamma^2$,
  and we summarize all these relations here:
\begin{equ}[e:lambda-sigma-mu]
\phi=\frac{\gamma}{2} X\;,
 \qquad
 \sigma = \sqrt\pi\gamma\;,
  \qquad
\lambda = \pi \mu\gamma^2\;.
\end{equ}
In this way we have
\begin{equs}[e:match-S]
 \frac{\pi\gamma^2}{2} \cdot \frac{1}{4\pi} \int_\Sigma \Big( |\nabla X|^2 + 2K_0 Q X+  4\pi \mu e^{\gamma X} \Big) \area_0
 &=
 \int_\Sigma \Big( \frac12 |\nabla \phi|^2 + K_0 \phi + \frac{\lambda}{2} e^{2\phi} \Big) \area_0 
 \\
& \mbox{where }\quad Q  =\frac{2}{\gamma} \;.
 \end{equs}

\paragraph{Insertions and Seiberg bounds.}

As mentioned above,
of great interest in Liouville CFT is the insertions of vertex operators 
\[
Z^{-1} \int \prod_{i=1}^n \Wick{e^{\alpha_i X(x_i)}}  \,e^{-S(X)} DX 
\]
for $n$ fixed points $x_i$ on the Riemann surface and $n$ real parameters $\alpha_i$
satisfying  the so called Seiberg bounds: $\sum_{i=1}^n \alpha_i > 2Q$ and $\alpha_i<Q$ for each $i$.
The corresponding stochastic dynamic -- by a similar calculation as in the proof of Theorem~\ref{thm:LIBP} and similar argument as in Remark~\ref{rem:invariance}
--
would be a formal equation of the following form
\[
\partial_t g=-2R_g-2\lambda g+2\sigma\xi_g g +\sum_{i=1}^n \alpha_i \delta_{x_i} \;,
\]
where each $\delta_{x_i}$ is a Dirac mass at $x_i$; or formally in terms of the area form
\begin{equ}
\partial_t \area_g(t) = 2\Lap \phi (t) \, \area_0 - 2\lambda\, \area_g(t) + 2 \sigma \xi_g \,\area_g(t)+\sum_{i=1}^n \alpha_i \delta_{x_i}  \;.
\end{equ}
When $\sigma=0$ (deterministic case) such equations appear in the context of metrics with conical singularities (see for instance \cite{phong2014ricci}); see Section \ref{Sec:insert} for further discussions.

\paragraph{Langevin dynamic /  stochastic quantization.}

The dynamic \eqref{e:SRicci} is expected to be symmetric with respect to the measure induced by Liouville CFT \eqref{e:match-S}, see Section~\ref{sec:diffusion}, as consistent with our motivation that  the dynamic  \eqref{e:SRicci} is viewed as the Langevin dynamic \eqref{e:langevin} of the Liouville CFT.
As in our proof,
the symmetry of the dynamic \eqref{e:SRicci} with respect to  Liouville CFT 
comes from the Dirichlet form theory, and in particular, by \cite[Eq.~(4.7.5)--(4.7.6)]{Fukushima2011} if one further has that the Dirichlet form is conservative \footnote{namely the associated Markovian semigroup $T_t$ has $T_t 1 =1$ a.e.} then
the dynamic \eqref{e:SRicci} is stationary and has  the Liouville CFT as an invariant measure.
 We will discuss in Section~\ref{sec:solution-Dirichlet} and Section~\ref{sec:extensions}
if the process is conservative or absorbed to zero  under different conditions. 

Constructing Langevin flows in the sense of \eqref{e:langevin}
for quantum field theory measures is also generally referred to as Parisi-Wu stochastic quantization \cite{ParisiWu},
and has drawn much attention in the recent years:
see e.g. \cite{Regularity,CC,Kupiainen} for the $\Phi^4_3$ model, 
\cite{MR3452276,SG8pi} for sine-Gordon model,
 \cite{hairer2016motion,bruned2019geometric}
for the loop measure on manifolds,
\cite{YM2020} for 2D Yang-Mills model,
and many other references therein.
In particular, we remark that  \cite{Garban2018} recently also studied a (different) Langevin flow for Liouville conformal theory: 
note that the gradient flow of $S(g_0,\phi)$
 with respect to the \underline{``flat"} metric $L^2(\area_0)$ would be
\begin{equ}
\partial_t\phi =\Delta \phi  -K_0 -\lambda e^{2\phi}
\;,
\end{equ}
and a stochastic version of this equation
(namely, this equation plus the space-time white noise with respect to the flat metric)
was studied in \cite{Garban2018}, and is called ``dynamical Liouville equation''; this is pursued later by 
\cite{hoshino2019stochastic,oh2019parabolic}.
We also remark that the significance of choosing a nontrivial metric (the ${\bf g}$ in  \eqref{e:langevin}) which is used to determine the gradient and the noise is  also observed in  other contexts. For instance, in aforementioned  \cite{hairer2016motion,bruned2019geometric},
one considers the space of all loops in a Riemannian manifold $M$, and at each given loop $u: S^1\to M$, an inner product on the ``tangent space'' at $u$ is defined in terms of the Riemannian metric of $M$ at all the points $u(x)$, see e.g. \cite[Sec.~1.1]{hairer2016motion}, and that yields a particular case of the Eells-Sampson Laplacian \cite{eells1964harmonic} appearing in the study of harmonic mappings in geometry (except loops are one-dimensional while in geometry harmonic mappings are of more interest in higher dimensions), and a noise compatible with this metric. Similarly in Yang-Mills model \cite{YM2020}, the gradient and noise need to respect the inner product one chooses on the Lie algebra.

 \subsection{Main result}
 
Our main result is a construction of weak solution to the equation \eqref{e:Ricci-A} 
 for the area measure $\area$. First, we need to formulate a notion of weak solution.
 By the calculation  \eqref{e:intrin-noise}, 
 we expect that given a suitable test function $f$, one should have the following one-dimensional projected stochastic equation
 \begin{equ}[e:1d-proj]
d\int_{\TT}  f \, \area_g
=2\left(\int_{\TT}   f\Delta\phi \area_0
-\lambda\int_{\TT}  f\, \area_g\right)dt
+2\sigma\left(\int_{\TT}  f^2\area_g\right)^{\frac12}d\beta_t \;.
\end{equ}
 Here $(\beta_t)$ is a one-dimensional standard Brownian motion. 

To formulate our result, 
let ${\mc M}_1(\TT)$ be the space of Borel probability measures on $\TT$ and ${\mc M}(\TT)$ be the space of finite positive Borel measures, equipped with the metrizable topology of weak (vague) convergence. Let 
\[
\MeaSp:={\mc M}(\TT)\setminus\{0\}\;.
\]
For $A\in \MeaSp$ and a function $f$ on $\TT$ we write by $A(f)$ 
the integral of $f$ with respect to the measure $A$.
We write $A(\TT) = A(1)$. 
Note that $\MeaSp$ is locally compact and $\MeaSp $ is homeomorphic to ${\mc M}_1(\TT)\times (0,\infty)$ via $A\mapsto (A/V,V)$, where $V=A(\TT)$ is the total measure of the torus $\TT$ under the measure $A\in \MeaSp$.  For an area form $\area$ we 
view it as a measure and
write $\area(f):=\int f \area$.

With $\MeaSp$ as our state space, we write $\MeaSp_\partial = \MeaSp \cup \{\partial\}$ (where the extra point $\partial$ is the ``cemetery'' state);
and we call a quadruple 
$\{\Omega,\mc F, (A_t)_{t\ge 0}, (P_z)_{z\in\MeaSp} \}$
 a Markov process on  $\MeaSp$ with time parameter $[0, \infty)$ if %the following conditions are satisfied:
$\{\Omega,\mc F, (A_t)_{t\ge 0}, (P_z)_{z\in\MeaSp} \}$ is a Markov process with state space $\MeaSp_\partial$  for each $x\in \MeaSp_\partial$, $P_x (A_t \in E)$ is measurable in $x\in \MeaSp$ for every $t\ge 0$ and every Borel set $E\subset \MeaSp$, $P_\partial (X_t =\partial)=1$ for all $t\ge 0$ and finally $P_x (X_0=x)=1$ for all $x\in \MeaSp$.
 
 \begin{Thm} \label{thm:main}
 For $  \sigma<\sigma_{L^1} = 2 \sqrt{\pi}$,
 there exists a Markov diffusion process
  $\mathbf A=\{\Omega,\mc F, (A_t)_{t\ge 0}, (P_z)_{z\in\MeaSp} \}$ on the space $\MeaSp$, such that for any smooth function $f$  and quasi-every $z\in \MeaSp$, $A_t(f)$ satisfies the following SDE
  \begin{equ} [e:1d-proj-theo]
d   A_t(f)
=2\Big(  \area_0( f\Delta\phi_t )
-\lambda    A_t ( f )\Big)dt
+2\sigma\left( A_t(f^2)\right)^{\frac12}d\beta_t^f \;,
\qquad 
A_0(f)=z(f) \;,
\end{equ}
where $\forall t>0$, $\phi_t=\mathbf M^{-1} A_t$ a.s. and
$\beta^f $ is a  one-dimensional standard Brownian motion.
 \end{Thm}

In the theorem, quasi-everywhere refers to the Dirichlet form and in particular implies $\m$-almost everywhere, where $\m$ is the law of a GMC (a measure on $\MeaSp$). Moreover for $t>0$ and q.e. $z$ the law of $A_t$ under $P_z$ is absolutely continuous w.r.t. $\m$ (see Section \ref{sec:diffusion}). Remark that $\phi_t$ only appears through
$$A_t(f)-A_0(f)=2\int_0^t\int_\Lambda f\Delta\phi_t\omega_0dt+\cdots$$
in \eqref{e:1d-proj-theo} (written here in integral form), so that it is enough to have it defined for (Lebesgue) almost every $t$ (see the discussion at the end of the proof of Lemma \ref{lem:identify-N}).

We say that a Markov  process
  $\mathbf A=\{\Omega,\mc F, (A_t)_{t\ge 0}, (P_z)_{z\in\MeaSp} \}$ on the space $\MeaSp$ 
 is {\it a weak solution} to the equation \eqref{e:Ricci-A} 
 if  for any smooth function $f$  and quasi-every $z\in \MeaSp$, $A_t(f)$ satisfies the  SDE \eqref{e:1d-proj-theo}.
 Theorem~\ref{thm:main} thus states the existence of a weak solution to   \eqref{e:Ricci-A}.

An immediate consequence of the theorem is that the total area $A_t(1)$ satisfies a stochastic ordinary differential equation:
\[
dA_t(1)=2\sigma\sqrt{A_t(1)}d\beta_t -2\lambda A_t(1)dt
\]
and therefore (see 
Corollary~\ref{cor:absorb} below)
the process $A_t(1)$ is a.s. absorbed at 0 in finite time.
 
\begin{Rem}
A next goal would be  to construct a (coupled) process $(\phi_t,A_t)$ where the component $\phi$ takes values in the space of Schwartz distributions $H^{-\eps}(\TT)$ and the component  $A$ takes values in the space of Borel measures,
such that for each $t>0$, $\phi_t $ is absolutely continuous w.r.t. the Gaussian free field, and the process $(A_t)$ is such that $A_t= \Wick{e^{2\phi_t} \area_0 } $ with
the  one-dimensional projection \eqref{e:1d-proj}.
See Section~\ref{sec:open-probs} for further discussions.
\end{Rem}

We will construct a weak solution using the theory of infinite-dimensional Dirichlet forms.
This is a general machinery to construct weak solutions of stochastic equations
which have explicit invariant (or at least symmetrizing) measures. We will frequently refer  to the book \cite{Fukushima2011}
when implementing this formalism. 
Among many applications of Dirichlet forms in constructing  weak solutions of stochastic equations, we mention 
several very recent ones that are to some extent related with our problem. In the context of stochastic quantization,
for the theory of Dirichlet forms is recently applied to the $\Phi^4$ model in three space dimensions by \cite{MR3858906,MR3631399}, see \cite{MR3828197} for a survey of these results (earlier result on this model in two space dimensions was obtained by \cite{AlbRock91}). 
Another instance of stochastic version of a geometric flow
 by \cite{rockner2017stochastic,chen2018stochastic} who considered a manifold-valued stochastic heat equations (and the strong solution has been also constructed, by \cite{bruned2019geometric}).
 We also remark that the theory of Dirichlet forms
was also recently exploited in the study of Liouville Brownian motions \cite{MR3531686,MR3272329,MR3531710} which is closed related with the Liouville measure we are concerned in this paper.

\begin{Rem}
In our  notation, $\area_g  = e^{2\phi}\area_0$ refers to the area form if $\phi$ is smooth or the ``formal'' area form if $\phi$ is rough.  $\M$ refers to the renormalized area form, i.e. GMC. 
Finally we will often denote a   generic element in $\MeaSp$
by the notation $A$.
\end{Rem}

\subsection*{Acknowledgements}

{\small
We would like to thank Rongchan Zhu, R\'emi Rhodes, Vincent Vargas and Steve Zelditch for useful discussions. JD gratefully acknowledges the  support by the NSF grant DMS-1308476 and DMS-1512853.
HS gratefully acknowledges the  support by the NSF grants DMS-1712684 / DMS-1909525 and DMS-1954091.
}

\section{Integration by parts}
\label{sec:IBP}

A key step of implementing the machinery  of Dirichlet forms 
is a proof of an integration-by-parts formula.
At first glance the form of integration-by-parts formulas we will provide below (with respect to both a Gaussian free field measure $\mu$ and a Liouville CFT measure $\nu$)
is similar to \cite[Theorem~5.3]{AlbeverioHK74}; we quickly summarize the setup in the latter  case
in Lemma~\ref{lem:DG-Classical} below.
 But the main difference comparing with our case is that
the functional  therein
is assumed to be ``cylindrical'' (or so-called ``finite-dimensional base'' therein), that is, of the form
$G(\phi)=q(\int f_1 \phi \area_0,\dots,\int f_k \phi  \area_0)$.
The Gaussian integration-by-parts formula for such cylindrical functionals (see \cite[Lemma~5.2 and (5.23)]{AlbeverioHK74}) boils down to finite-dimensional Gaussian integration by parts essentially because $\int f_i \phi \area_0$ are Gaussian.
For our problem however, we need to consider a different class of functionals (see Definition~\ref{def:mcC} below), tailored to this specific situation.
The proof of integration by parts is based on 
shift covariance of Liouville measure together with the Cameron-Martin formula.

Denote by $H^\gamma = H^{\gamma,2} (\TT)$ the Sobolev Hilbert space.
In the sequel we write $\Phi:= H^{-\eps}$ where $\eps$ is a fixed, small positive real number. 
A general element $\phi \in \Phi$ can be uniquely decomposed as $\phi=m+\phi_0$, with $\phi_0$ zero-mean and $m\in \R$. 
We have the measure on $\Phi$
\[
d\hat\mu(\phi)=dm\otimes d\mu(\phi_0)
\]
where $dm$ is the Lebesgue measure on $\R$ and $\mu=\mu_\sigma$ is the Gaussian Free Field probability measure on zero-mean fields, for the covariance operator $\frac{\sigma^2}2(-\Delta)^{-1}$ (where $(-\Delta)^{-1}$ denotes the zero-mean Green kernel); $\sigma<2\sqrt\pi$ is fixed. The $\sigma$-finite measure $dm\otimes d\mu(\phi_0)$ is (up to multiplicative constant) the natural interpretation of the path integral measure
\begin{equation}\label{eq:pathint}
e^{-\sigma^{-2}\int_{\TT}|\nabla\phi|^2\omega_0}{\mc D}\phi \;.
\end{equation}

Denote by $H= H^1 (\TT)$ the Cameron-Martin space.
For any $A\in \MeaSp$, we denote by $L^2(A)$ the $L^2$ space with underlying measure $A$ on $\TT$. Recall that smooth functions are dense in $L^2(A)$, which is separable.

Recall that for any vector space $\Phi$, a function $G$ on (generally an open set of) $\Phi$, and elements
$\phi,h \in \Phi$,
we call $D_h G (\phi)$ the Fr\'echet derivative \footnote{Also called the Gateaux directional derivative.}
of $G$ in the direction $h$ at $\phi$
if the limit $D_h G (\phi) = \lim_{t\to 0} \frac{1}{t}(G(\phi+th) - G(\phi))$ exists.
Here, we have implicitly identified  the tangent space of the domain of $G$ at $\phi$ with  the vector space $\Phi$ itself (they share the same linear structure     so that we can simply add $\phi$ with a ``tangent vector'' $th$); however, as we discuss gradients of $G$, an inner product needs to be specified on 
the tangent space at each $\phi$. Given such an inner product $\langle \;,\; \rangle$ at each point $\phi$, the $\langle\;,\; \rangle$-gradient $DG(\phi)$ is then defined by the element in $\Phi $ such that $\langle DG(\phi) , h \rangle = D_hG(\phi)$, as long as such an element exists. As we will see below, the particular choice of this inner product will play an important role.

We start by recalling the classical results on gradients of test functionals and integration by parts for the Gaussian free field as in \cite{AlbeverioHK74} for the sake of comparison, without proof.

\begin{Lem} \label{lem:DG-Classical}
(The ``classical'' case.)
Let $G(\phi)= q(\int f_1 \phi \area_0,\dots,\int f_k \phi  \area_0)$ where $q: \R^k \to \R$ is a compactly supported $C^2$ function and the $f_i$'s are in $H^{-1}$. 
Then  $G$ has bounded Fr\'echet derivative in Cameron-Martin directions and
one has
\begin{equ}[e:DhG-Classical]
D_h G (\phi) 
= \sum_{i=1}^k \partial_i q\Big(\int f_1 \phi \area_0,\dots,\int f_k \phi  \area_0\Big)
\cdot \int ( f_i h )\area_0
\end{equ}
for any Cameron-Martin direction $h\in H$.
The $L^2(\area_0)$-gradient of $G$ is 
 characterized by 
$\langle D G (\phi),h\rangle_{L^2(\area_0)}=D_hG(\phi)$ for all $h\in H^1$ a.s., 
and is given by
\begin{equ}[e:DG-Classical]
D G (\phi)=  \sum_{i=1}^k \partial_i q\Big(\int f_1 \phi \area_0,\dots,\int f_k \phi  \area_0\Big)
 \,f_i .
\end{equ}
For such functionals $G$, we have the following Gaussian integration by parts
\begin{equ}[e:GIBP-Classical]
 \frac{\sigma^2}2 \int D_hG(\phi)\hat\mu(d\phi)
=\int G(\phi) \langle \nabla h, \nabla \phi \rangle \hat\mu(d\phi)
\end{equ}
where  $d\hat\mu(\phi)=dm\otimes d\mu(\phi_0)$ as defined above.
\end{Lem}
As usual, $\langle\nabla h, \nabla \phi\rangle$ is defined everywhere if $h\in H^{2+\eps}$ and a.e. (via Paley-Wiener) if $h\in H^1$.

\subsection{Test functionals}

We now define a class $\mc C$ of test functionals on $\Phi$ suitable for our purposes. To this end we recall the GMC mapping 
\begin{align*}
\Phi&\longrightarrow\MeaSp\\
\phi&\longmapsto M_\phi=\Wick {e^{2\phi}\omega_0}
\end{align*}
which is defined $\hat\mu$-almost everywhere, see Section \ref{Subsec:GMC}.

\begin{Def}\label{def:mcC}
Let $\TestSe$ be the space of functionals on $\Phi$
of the form
\begin{equ}[e:C-form]
G(\phi)=q( \M(f_0) , \M(f_1) ,\dots, \M( f_k))  \qquad \mbox{for } \phi\in \Phi
\end{equ}
such that $q: \R^{k+1} \to \R$ is a $C^2$ function
 and $f_i$ are smooth functions with $f_0\equiv1$.

 Let $\TestSp\subset \TestSe$ be the space of functionals $G\in \TestSe$ on $\Phi$
 such that there exists a   compactly supported $q: \R^{k+1} \to \R$ i.e.
\begin{equ}[e:zero-th-co]
\mbox{Supp}(q) \subset (\eps,\eps^{-1})  \times Q  \qquad (\mbox{for some } \eps\in(0,1) \mbox{ and } Q\subset \R^k \mbox{ compact})
\end{equ}
   so that $G$ is 
of the form \eqref{e:C-form}. 
In particular, elements of $\TestSp$ are bounded.
\end{Def}

We now compute Fr\'echet derivatives and gradient of functionals in $\TestSe$.
Denote by $C^0(\TT)$ the space of continuous functions on $\TT$.

\begin{Lem} \label{lem:DG}
Let  $G\in \TestSe$ be of the form \eqref{e:C-form}. 
Then  $G$ has the Fr\'echet derivative
\begin{equ}[e:DhG]
D_h G (\phi) 
= 2 \sum_{i=0}^k \partial_i q (\M(f_0),\dots,\M( f_k ))
\cdot \M( f_i h )
\end{equ}
for any $h\in C^0(\TT)\cap H$.
In particular $D_h \M( f)=2 \M(fh)$. 

The $L^2(\M)$-gradient of $G$ 
is characterized by 
\begin{equ}[e:gradchar]
\langle D G (\phi),h\rangle_{L^2(M_\phi)}=D_hG(\phi) 
\end{equ}
$\hat\mu$-a.e. for any $h\in C^0(\TT)\cap H$, and is
 given by
\begin{equ}[e:DG]
D G (\phi)= 2 \sum_{i=0}^k \partial_i q (\M( f_0 ),\dots, \M( f_k))
 \,f_i.
\end{equ}
Finally, if we further have $G\in \TestSp$, then  the Fr\'echet derivative $D_h G$ is bounded
 for all $\phi\in\Phi$.
\end{Lem}

\begin{proof}
We first remark that for a fixed $\phi$, $\M( f_i h ) < \infty$ 
so that the right-hand side of \eqref{e:DhG} is well-defined.
This is because $f_i$ and $h$ are continuous on $\TT$ thus bounded,
and $\M$ is finite.

By the shift property \eqref{e:shiftcov},
\[
M_{\phi+ t h } (f)= \big(e^{ 2 t \cdot h} M_\phi\big)(f) 
= \int_{\TT} f(x) e^{ 2 t \cdot h(x)} \M(dx)
\quad \mbox{a.e.}
\]
for any $h\in H$.
One then has
\begin{equs}
G& (\phi+th)    - G (\phi)
= q \left( \big(e^{ 2 t \cdot h} \M\big)(f_1)    ,   \dots, \big(e^{ 2 t \cdot h} \M\big)(f_k)  \right)
-q(\M(f_1),\dots, \M (f_k ))
\\
&  =\sum_{i=1}^k q \Big(\M(f_1),\dots,\M(f_{i-1}),  
	\big(e^{ 2 t \cdot h} \M\big)(f_i) ,\dots, \big(e^{ 2 t \cdot h} \M\big)(f_k)\Big)
	\\
	 &\qquad \qquad \qquad \qquad  -   q \Big(\M(f_1),\dots,\M(f_{i}),  
	\big(e^{ 2 t \cdot h} \M\big)(f_{i+1}) ,\dots, \big(e^{ 2 t \cdot h} \M\big)(f_k)\Big)
%\cdot 2  \int f_i h  e^{2th}  dA_\phi
\\
& = t\cdot \sum_{i=1}^k \partial_i q \Big(\M(f_1),\dots,\M(f_{i-1}),  
	\big(e^{ 2 t_\star \cdot h} \M\big)(f_i) ,\dots, \big(e^{ 2 t \cdot h} \M\big)(f_k)\Big) 
	\\
&
\qquad\qquad\qquad\cdot 
\frac{d}{dt} \Big|_{t=t_\star} 
	 \int_{\TT} f_i(x) e^{ 2 t \cdot h(x)} \M(dx)
\end{equs}
for some $t_\star \in [0,t]$.
By the aforementioned boundedness of $h$  we can bound $e^{ 2 t \cdot h(x)} $ by a constant, 
and thus by dominated convergence theorem
one has that $\frac{d}{dt}   \Big|_{t=0}  G (\phi+th) $
is equal to the right-hand side of \eqref{e:DhG}.

Once we have the Fr\'echet derivatives $D_h G$,
the gradient $DG$ then exists and is unique by Riesz representation theorem, 
since the space $C^0(\TT)\cap H$ is dense in  $L^2(M_\phi)$.
Indeed, the GMC measure $M_\phi$ is a Radon measure that is inner and outer regular,
so for any Borel set $A\subset \TT$ there exist an open set $U$
and a compact set $K$ such that 
$K\subset  A\subset U$ and $\M(U\backslash K)$ is arbitrarily small,
then by Urysohn's lemma one obtains a 
continuous function which is supported on $U$ and equal to $1$ on $K$,
thus approximates the characteristic function of $A$ in the $L^2(M_\phi)$ topology.
The simple functions, namely  linear combinations 
of the characteristic functions are then dense in  $L^2(M_\phi)$ by construction
of integrals with respect to the GMC $\M$.

Regarding the identity \eqref{e:DG},  with $DG$ in  \eqref{e:DG}  one can immediately check that
\[
\M(h \cdot DG) = D_hG \;,
\]
namely \eqref{e:gradchar} holds.

If $G\in \TestSp$, $D_h G$ is  bounded. 
In fact, since $f_i,h$ are continuous, $|\M(f_i h )|\le C \M(1)$;
and by the fact that 
 $\partial_i q=0$ when  $\M(1)\notin (\eps,\eps^{-1})$, one obtains the boundedness of
\[
\partial_i q (\M(1),\M(f_1),\dots,\M( f_k ))  \M(1) \;.
\]
\end{proof}

Obviously, $D_h G$ and $DG$ do not depend on the representation 
 \eqref{e:C-form}, namely if  $G(\phi)$ is equal to $\tilde q( \M(\tilde f_0) ,\cdots, \M(\tilde  f_\ell))$ for some other functions  $\tilde q$  and $\{\tilde f_1,\cdots,\tilde f_\ell\}$ and $\ell\ge 0$, then 
 the right-hand side of \eqref{e:DhG} or \eqref{e:DG} 
 with $q$  and $\{f_1,\cdots, f_k\}$
 replaced by $\tilde q$  and $\{\tilde f_1,\cdots,\tilde f_\ell\}$ remains identical. 
Indeed we showed that $D_hG=\lim_{t\searrow 0}\frac{G(\phi+th)-G(\phi)}t$ a.e., and $DG$ is characterized by the $D_hG$'s.

We also note that the Leibniz rule holds:
\begin{equ}[e:Leibniz]
D_h (GH) = (D_h G) H + G (D_h H) \qquad \forall \; G,H\in\TestSe \;.
\end{equ}
 Indeed, for $G(\phi)=q(\M(f_1) ,\cdots, \M( f_k))  $ and $H(\phi)=p(\M(g_1) ,\cdots, \M( g_\ell))  $,
 we have $(GH)(\phi) = r (\M(f_1) ,\cdots, \M( f_k),\M(g_1) ,\cdots, \M( g_\ell))$
 where $r(x_1,\cdots,y_\ell) $ equals $ q(x_1,\cdots,x_k) p(y_1,\cdots,y_\ell)$.
 So by the formula  \eqref{e:DhG} we have that $D_h(GH)(\phi)$ equals
 $2\sum_{i=1}^k \partial_{x_i} r \cdot \M(f_i h) + 2\sum_{j=1}^\ell \partial_{y_j} r \cdot \M(g_jh) $
which is equal to the right-hand side of \eqref{e:Leibniz}.

\subsection{Proof of integration by parts}

\begin{Lem} \label{lem:GIBP}
Let $\mu$ be the law of a mean zero GFF $\phi_0$ on $\TT$, with covariance operator $\frac {\sigma^2}2(-\Lap)^{-1}$, $\phi=\phi_0+m$, and $d\hat\mu(\phi)=dm\otimes d\mu(\phi_0)$.
Then we have the following Gaussian integration by parts
\begin{equ}[e:GIBP]
 \frac{\sigma^2}2 \int D_hG(\phi)\hat\mu(d\phi)
=\int G(\phi) \langle \nabla h, \nabla \phi \rangle \hat\mu(d\phi)\;,
\qquad
\forall G\in \TestSp %which is bounded, 
\end{equ}
and $D_h$ is the Fr\'echet derivative in the Cameron-Martin direction $h\in C^0(\TT)\cap H$.
\end{Lem}

\begin{proof}
By boundedness of $D_h G$ from Lemma~\ref{lem:DG}
and boundedness of  $G$ by definition, both sides of \eqref{e:GIBP} are well-defined.
Recall that $H=H^1$ is the Cameron-Martin Hilbert space, endowed with
$\langle \cdot,\cdot \rangle_H$.
 For $h\in H$ with mean zero and $t\in \R$, one has the Cameron-Martin formula
\[
\frac{d T^{th}_* \mu} {d\mu}
=
\exp \Big( t \langle \phi_0, h\rangle_H - \frac{t^2}{2} \|h\|_H^2
\Big)
\]
where $T^{th}_* \mu$ denotes the push-forward measure of $\mu$
in the direction $th$.
Let $G$ be as assumed above.
One then has
\begin{equs} [e:to-be-diff-t]
\int_\R  \int   &  G(m+\phi_0 + th)\mu(d\phi_0)dm
\\ &=
\int_\R  \int G(m+\phi_0)  \exp \Big( \frac{2}{\sigma^2}t \langle \phi_0, h\rangle_H - \frac{t^2}{\sigma^2} \|h\|_H^2
\Big) \mu(d\phi_0) dm\;.
\end{equs}
Since $\langle \phi_0, h\rangle_H = \int\nabla h\cdot\nabla \phi_0=\langle \phi, h\rangle_H $,
it remains to differentiate the above identity in $t$ at $t=0$
using dominated convergence theorem.

Again since $G$ has bounded Fr\'echet derivative
by Lemma~\ref{lem:DG},
we have that  
 differentiating the l.h.s. of \eqref{e:to-be-diff-t} w.r.t. $t$ at $t=0$ using dominated convergence yields 
l.h.s. of \eqref{e:GIBP}.

For  the r.h.s. of \eqref{e:to-be-diff-t}, for sufficiently small $t>0$ one has
\begin{equs}
 \Big|\frac{d}{dt} \exp \Big( t \langle \phi_0, h\rangle_H - \frac{t^2}{2} \|h\|_H^2
\Big)\Big|
&=   \Big| \langle \phi_0, h\rangle_H - t \|h\|_H^2
\Big|
\exp \Big( t \langle \phi_0, h\rangle_H - \frac{t^2}{2} \|h\|_H^2
\Big) 
\\
&\le C \exp \Big(a \Big| \langle \phi_0, h\rangle_H \Big|\Big)
\end{equs}
for some  constants $a,C>0$.

By \eqref{e:exp-mom} in Lemma~\ref{Lem:sigmafin}  and boundedness of $G$,
one has that 
\[
G(m+\phi_0) \cdot \exp \Big(a \Big| \langle \phi_0, h\rangle_H \Big|\Big)
\]
 is $\hat\mu$ (namely $m\otimes \mu$)-integrable over fields 
such that 
$M_\phi(\TT) \in [\eps,\eps^{-1}]$, where $\eps$ is the constant arising from the specification of the support in the first coordinate of $G$, see \eqref{e:zero-th-co}.
For $M_\phi(\TT) \notin [\eps,\eps^{-1}]$, $G(m+\phi_0) $ simply vanishes by assumption.
Therefore dominated convergence applies and the derivative of the r.h.s. of \eqref{e:to-be-diff-t} w.r.t. $t$ at $t=0$ is
\[
\int_\R\int G(\phi_0+m) \langle \nabla h, \nabla \phi_0 \rangle \mu(d\phi_0)dm\;.
\]
This is the r.h.s. of \eqref{e:GIBP}. 
We thus showed that \eqref{e:GIBP} holds if $h$ has mean zero. If $h$ is constant, both sides of \eqref{e:GIBP} are zero (by translation invariance of $dm$); this concludes by linearity.
\end{proof}

Let
$\nu$ be the Liouville CFT measure on $H^{-\eps}$ given by (recall the conventions about the parameters  %\eqref{e:gamma-sigma} and
\eqref{e:lambda-sigma-mu})
\begin{equ}[e:def-nu]
d\nu(\phi)=  \exp\Big(  -\frac\lambda{\sigma^2}
\M(\TT)\Big)dm\otimes d\mu(\phi_0)
\end{equ}
where $\phi=m+\phi_0$. % and $Z:=\int \exp\Big(-\frac\lambda{\sigma^2} \M(\TT)\Big)dm\otimes d\mu(\phi_0) $.
The Liouville CFT measure $\nu$ has been rigorously constructed 
by \cite{MR3450564,guillarmou2016polyakov} with suitable insertions
of vertex operators. Without insertions on the torus $\nu$ is not normalizable
because the integral of $\nu$ would diverge as $\phi \to -\infty$ (so that $\M(\TT)\to 0$). 
Here we do not consider insertions but 
instead we verify that $\nu$ is $\sigma$-finite, see Lemma~\ref{Lem:sigmafin}.

We start with a basic integrability result. Recall that $\nu$ depends on the parameters $\sigma<\sigma_{L^1}$ and $\lambda>0$.
\begin{Lem}\label{Lem:sigmafin}
$\nu$ is $\sigma$-finite; more precisely, for any $\eps\in (0,1)$, 
\[
\nu(\{\phi: \eps<M_\phi(\TT)<\eps^{-1}\}) <\infty \;.
\]
Moreover, for $f\in H$,
\begin{equ}[e:int-dfdphi-restrict]
\phi  \mapsto \langle f,\Lap\phi\rangle\ind_{[\eps,\eps^{-1}]}(M_\phi(\TT))
\end{equ}
is  in  $L^p(\hat\mu)$ and $L^p(\nu)$ for all $p\in [1,\infty)$, 
and for $a>0$  
\begin{equ}[e:exp-mom]
%\phi \mapsto \exp \big( a  \|\phi_0\|_H^2  \big) \ind_{[\eps,\eps^{-1}]}(M_\phi(\TT)) 
\phi \mapsto \exp \big( a \big| \langle \phi_0,f\rangle_H\big|  \big) \ind_{[\eps,\eps^{-1}]}(M_\phi(\TT)) 
\end{equ}
is  integrable with respect to $\hat\mu$ and $\nu$. 
\end{Lem}
\begin{proof}
Note that the Gaussian measure $\mu$ on zero-mean fields $\{\phi_0\}$ on the torus is a probability measure. %and 
By the shift property of the GMC, we have
\begin{align*}
\int_{-\infty}^\infty  \int_{\Phi} &
\ind_{\eps<M_\phi(\TT)<\eps^{-1}}d\mu(\phi_0)dm
=
\int_{-\infty}^\infty\int_{\Phi}
\ind_{\eps e^{-2m}<M_{\phi_0}(\TT)<\eps^{-1}e^{-2m}}d\mu(\phi_0)dm\\
&\leq
\sum_{k=-\infty}^\infty\int_{k|\log\eps|}^{(k+1)|\log\eps|} 
\mu\Big( \{\eps^{2k+3}<M_{\phi_0}(\TT)<\eps^{2k-1}\}\Big) dm
\\
&\leq C'\sum_{k=-\infty}^\infty
	\mu\Big( \{\eps^{4k+3}<M_{\phi_0}(\TT)<\eps^{4k-1}\}\Big) \le C
<\infty
\end{align*}
where the constants $C',C$ depend on $\eps$. 
Here note that for  the intervals $(\eps^{2k+3},\eps^{2k-1}) \subset (0,\infty)$
in the second line only two adjacent intervals overlap, and
 in the last step the intervals $(\eps^{4k+3},\eps^{4k-1}) \subset (0,\infty)$ are non-overlapping
 so that we can make use of the fact that $\mu$ is a finite measure. 
This shows that
\[
\hat\mu(\{\phi: \eps<M_\phi(\TT)<\eps^{-1}\}) \le C \;.
\]
Together with $\exp\Big(  -\frac\lambda{\sigma^2}
\M(\TT)\Big)\le 1$ this
 gives the first claim.

From Proposition 3.5 and 3.6 in \cite{MR2642887} we have the following positive and negative moment estimates for the total mass of a GMC:
\[
\int_\Phi (M_{\phi_0}(\TT))^pd\mu(\phi_0)<\infty
\]
for all $p<0$ and for some $p=p(\sigma)>1$. In particular, by Markov's inequality, for $x$ large,
\[
\mu(\{M_{\phi_0}(\TT)\geq x\})=O(x^{-1}),\quad \mu(\{M_{\phi_0}(\TT)\leq x^{-1}\})=O(x^{-1})\;.
\]
For $f\in H$, $\langle f,\Lap\phi\rangle=\langle f,\Lap\phi_0\rangle$ %=\langle \Lap f,\phi_0\rangle$
 is Gaussian and hence has moments of all orders (under $\mu$). 
Then
\begin{equs}
\int_{-\infty}^\infty\int_{\Phi}  &
|\langle f,\Lap\phi\rangle|^p\ind_{[\eps,\eps^{-1}]}(M_\phi(\TT))
dmd\mu(\phi_0)
\\
&\leq\sum_{k=-\infty}^\infty\int_{k|\log\eps|}^{(k+1)|\log\eps|}  \int_{\Phi}
	|\langle f,\Lap\phi\rangle|^p\ind_{[\eps^{3+2k},\eps^{2k-1}]}(M_{\phi_0}(\TT))
dmd\mu(\phi_0)\\
&\leq C\sum_{k=-\infty}^\infty \left\| |\langle f,\Lap{\phi_0}\rangle|^p\right\|_{L^2(\mu)}
	\Big(\mu\{\eps^{3+2k}<M_{\phi_0}(\TT)<\eps^{2k-1}\}\Big)^{1/2}\\
&<\infty
\label{e:p-mom-cut}
\end{equs}
by Cauchy-Schwarz and the previous estimate.

The last statement \eqref{e:exp-mom} is proved in the same way 
given that $  \langle \phi_0,f\rangle_H $ is centered Gaussian random variable with variance
$\|f\|_H^2$,
so that
\[
\|   e^{a | \langle \phi_0,f\rangle_H | }   \|_{L^2(\mu)}
\le   \| e^{a  \langle \phi_0,f\rangle_H  }   \|_{L^2(\mu)}
+  \| e^{ -a  \langle \phi_0,f\rangle_H  }   \|_{L^2(\mu)}
 = 2 e^{a^2 \|f\|_H^2 } <\infty\;.
\]
%$\exp \big( a \|\phi_0\|_H^2  \big)$ 
The estimate \eqref{e:p-mom-cut} with $\left\| |\langle f,\Lap{\phi_0}\rangle|^p\right\|_{L^2(\mu)}$ replaced by 
$\|  e^{a  | \langle \phi_0,f\rangle_H | }   \|_{L^2(\mu)} $ 
 then shows that
  $ e^{a | \langle \phi_0,f\rangle_H|} \ind_{[\eps,\eps^{-1}]}(M_\phi(\TT))$
is integrable for  $a>0$ with respect to the Gaussian measure $\hat\mu$ and thus also $\nu$.
\end{proof}

\begin{thm} \label{thm:LIBP}
(Integration by parts for Liouville CFT $\nu$) For any  $G\in \TestSp$ and $h \in C^0(\TT)\cap H$
\begin{equ}[e:LIBP]
\int G(\phi)  \langle \nabla\phi, \nabla h \rangle d\nu(\phi)=\int \left(\frac{\sigma^2}2D_h G(\phi)
	- \lambda G(\phi)  \M(h)\right)d\nu (\phi) \;.
\end{equ}
\end{thm}

\begin{proof}
We first remark that all the three terms in \eqref{e:LIBP} are $\nu$-integrable.
Indeed, the left-hand side of  \eqref{e:LIBP} is finite, 
since by the assumption \eqref{e:zero-th-co}  one can bound $G$ by a constant times
$ \ind_{[\eps,\eps^{-1}]}(M_\phi(1)) $, and then we apply the integrability of \eqref{e:int-dfdphi-restrict}
in Lemma~\ref{Lem:sigmafin}. 
Regarding the right-hand side of  \eqref{e:LIBP},
by the formula \eqref{e:DhG} and the assumption \eqref{e:zero-th-co}  
we can bound 
$D_h G $  by $C\eps^{-1} \ind_{[\eps,\eps^{-1}]}(M_\phi(1))$ for some constant $C>0$,
which is again integrable by Lemma~\ref{Lem:sigmafin}. 
The same bound holds for $\lambda G(\phi)  \M(h)$ 
and thus is integrable too. 

To prove  \eqref{e:LIBP}, note that 
the left-hand side of \eqref{e:LIBP} equals 
\begin{equ}[e:applyGIBP]
\int G(\phi)  \langle \nabla\phi, \nabla h \rangle
e^{-  \frac\lambda{\sigma^2}  \M(1)}
 d\hat\mu(\phi)
 =
 \frac{\sigma^2}{2}
\int D_h \Big(G(\phi) e^{- \frac\lambda{\sigma^2} \M(1)}\Big)
d\hat\mu(\phi)
\end{equ}
where we applied
 Lemma~\ref{lem:GIBP}  to the functional $G(\phi) e^{-\frac\lambda{\sigma^2}\M(1)} \in \TestSp$. 

By Lemma~\ref{lem:DG}
\begin{equ}[e:Dhapplyto]
D_h  e^{- \frac\lambda{\sigma^2} \M(1)}
= -\frac{2\lambda}{\sigma^2}
\M(h)
e^{- \frac\lambda{\sigma^2} \M(1)} \;.
\end{equ}
Invoking this in \eqref{e:applyGIBP} and applying \eqref{e:Leibniz} we obtain the
right-hand side
 of 
\eqref{e:LIBP}.
\end{proof}

\section{Solution via Dirichlet forms} \label{sec:solution-Dirichlet}

The theory of Dirichlet forms is a general framework for constructing solutions to stochastic differential equations.
For convenience of readers from different backgrounds, we briefly recall the key notions in this theory.
In general, given a real Hilbert space $(H,\langle \;,\;\rangle_H)$, which we always think of as a space of the form $L^2(X,m)$ for some $\sigma$-finite measure space $(X,m)$, a non-negative definite symmetric bilinear form $\mathcal{E}$ defined on a dense set $\mathcal{D}[\mathcal{E}]$ of $H$ is called a symmetric form on $H$.
An inner product can be defined on the domain $\mathcal{D}[\mathcal{E}]$ so that $\mathcal{D}[\mathcal{E}]$ becomes a pre-Hilbert space: indeed, for each $\alpha>0$,
$\mathcal{E}_\alpha (u,v) := \mathcal{E} (u,v) + \alpha \langle u,v\rangle_H$ defined on $\mathcal{D}[\mathcal{E}]$ is again a symmetric form, and this gives a metric (which is equivalent for different $\alpha>0$).
\footnote{We note that in general,  $\mathcal{D}[\mathcal{E}]$ is not even a pre-Hilbert space with respect to  $ \mathcal{E} $ (that is, $\mathcal{E}_\alpha$ with $\alpha=0$).}
A symmetric form is called closed if  $\mathcal{D}[\mathcal{E}]$ is complete with respect to this metric (i.e. Cauchy sequences converge in  $\mathcal{D}[\mathcal{E}]$ under this metric), namely, $\mathcal{D}[\mathcal{E}]$ is actually Hilbert.
{\it A Dirichlet form}, by definition, is then  a symmetric form which is Markovian and closed.
We often call $(\mathcal{E}, \mathcal{D}[\mathcal{E}])$ a ``Dirichlet space''.

Given a symmetric form $\mathcal{E}$, and ``extension'' of $\mathcal{E}$ is just another symmetric form whose domain contains $\mathcal{D}[\mathcal{E}]$, and restricting to $\mathcal{D}[\mathcal{E}]$ the two symmetric forms are identical. A symmetric form $\mathcal{E}$ having a closed extension is equivalent with saying that $\mathcal{E}$ is closable.  We refer to \cite[Section~1.1]{Fukushima2011} for more discussion of these notions.

In a nutshell, the link connecting the theory of Dirichlet forms with Markov processes is 
that there is a one-to-one correspondence between the family of closed symmetric forms on $H$ and the family of non-positive definite self-adjoint operators $A$ (serving as generators of the processes) on $H$,
given by $\mathcal{E} (u,v) = \langle \sqrt{-A}u,\sqrt{-A} v\rangle_H$.
The standard notions in stochastic processes, such as the family of strongly continuous semigroups, and the family of strongly continuous resolvents, are then also in one-to-one correspondences with this family of generators.
Markovian property of forms is shown to be equivalent to the Markovian properties of the associated semigroups and resolvents.

The main result in the theory of Dirichlet forms states that 
as long as a given Dirichlet form $\mathcal{E}$ is ``regular'' (namely it has a ``core'' which is by definition a subset $\mathcal{C}$ of $\mathcal{D}[\mathcal{E}] \cap C(X)$
such that $\mathcal{C}$ is dense in $\mathcal{D}[\mathcal{E}]$ with  norm  $\mathcal{E}_\alpha$ and dense in the space of continuous functions $C(X)$ with uniform norm),
then there exists an $m$-symmetric Markov process on $X$ whose associated Dirichlet form (in the sense of the above correspondence) is $\mathcal{E}$; in particular it is a stationary process with respect to $m$.  Of course,
a Markov process associated with $\mathcal{E}$ would not be very interesting unless 
it has certain sample paths regularity; in this context, the process given by the theory will be a ``Hunt process'', which is strong Markov, right-continuous, and quasi-left continuous. An introduction to Hunt processes is given in the appendix of \cite{Fukushima2011}.
Finally, to view this process as a solution to the given stochastic differential equation,
one needs a type of semi-martingale decomposition, in this context called Fukushima decomposition,
which identifies the drift term and martingale / noise term in the given equation.

This concludes a brief sketch of the theory, and we will provide more precise references when using the above results. To summarize, 
the construction of the weak solution  via Dirichlet forms consists of three steps. 1. Showing closability of the Dirichlet form; \footnote{Here are below, we sometimes slightly abuse the terminology by saying ``closability of the Dirichlet form'' for ``closability of the symmetric form whose closed extension is thus a Dirichlet form''.} 2. Proving existence of Hunt process associated to the  Dirichlet form; 3. Proving the process solves the equation in certain sense.

\vspace{2ex}

Recall from Section \ref{Subsec:GMC} that there is an a.e. correspondence 
$\phi\leftrightarrow\M$. We denote by $\m$ the image measure of the Liouville CFT measure 
$\nu$
by the measurable map  $\mathbf M$
\begin{equs}[e:mapM]
\Phi &\to \MeaSp
\\
\phi &\mapsto \mathbf M(\phi)= \M
\end{equs}
 Then $\m=\mathbf M_*\nu$ is a Radon measure on $\MeaSp$, see Lemma \ref{Lem:sigmafin}. 
We denote by $L^2(\MeaSp,\m)$ the Hilbert space of square integrable $\m$-measurable functions on $\MeaSp$.
%In view of Lemma~\ref{lem:BSS}
We denote by $\mathbf M^{-1}$ an a.e. inverse measurable map to $\mathbf M$.
The spaces $L^2(\MeaSp,\m)$ and $L^2(\Phi,\nu)$ are isometric 
under the pull back map $\mathbf M^*=(\mathbf M^{-1})_*$.

Recall that $\MeaSp$ is locally compact while $\Phi$ is merely Polish.

We will first introduce a form on $\Phi$, and then, induce a  form on $\MeaSp$.
To this end we define the following class of test functions $\CX$ on $\MeaSp$: $\CX$ consists of test functionals $F:\MeaSp\rightarrow\R$ such that $F(A)=q(\int f_0dA,\dots,\int f_kdA)$
 for some smooth functions $f_0,\dots,f_k\in C^\infty(\TT)$ and some function $q$ as in Definition \ref{def:mcC} and satisfying \eqref{e:zero-th-co}.

 Let $C_0(\MeaSp)$ be the space of compactly supported continuous functions on $\MeaSp$ with uniform norm.
 
\begin{Lem}\label{lem:dense}
$\CX$ is dense in $C_0(\MeaSp)$, and is dense in $L^2(\MeaSp,\m)$. The space $\TestSp$ is dense in $L^2(\Phi,\nu)$. 
\end{Lem}

\begin{proof}
To prove that $\CX$ is dense in $C_0(\MeaSp)$, by the Stone-Weierstrass  theorem for locally compact spaces, 
it suffices to prove that $\CX$  is an algebra of functions which separates points in $\MeaSp$ and vanishes nowhere. $\CX$ is clearly an algebra and it vanishes nowhere: indeed for any $M\in \MeaSp$, 
recalling that $\MeaSp\simeq {\mc M}_1(\TT)\times (0,\infty)$ 
one has $M(1)=M(\TT) \in (0,\infty)$
so $F(M):= q(M(1)) \in\CX$ is not equal to $0$ for any function $q$ that does not vanish at $M(1)$.
It is also clear that  $\CX$  separates points in $\MeaSp$: indeed,
for 
$M_1 \neq M_2 \in \MeaSp$, there must exist $f$ smooth such that $M_1(f)\neq M_2(f)$,
thus $F(M):= q(M(1),M(f)) \in\CX$  with this function $f$ separates $M_1 $ and $ M_2$ for any choice of function $q$ which takes different values at $(M_1(1),M_1(f))$ and $(M_2(1),M_2(f))$.

Since $\m$ is a Radon measure on $\MeaSp$, by the same argument as in the proof of Lemma~\ref{lem:DG}, namely using inner and outer regularities of $\m$ with Urysohn's lemma,
one has that $C_0(\MeaSp)$ is dense in $L^2(\MeaSp,\m)$.

The fact that  $\TestSp$ is dense in $L^2(\Phi,\nu)$ follows immediately due to the aforementioned isometry.
\end{proof}

Clearly  $\CX\subset L^p(\m)$ for all $p<\infty$ by Lemma \ref{Lem:sigmafin}. Moreover, if $F\in\CX$, then $\tilde F=F\circ{\mathbf M}$ is in ${\mc C}$.

\subsection{Closability of the Dirichlet form}
\label{sec:Closability}

\begin{Def}
 For $F (\phi) = q (\M (f_1) ,\dots,\M( f_k)) \in \TestSp =: \mc D(\mc L)$, we define
\begin{equ} [e:def-L]
{\mc L}F(\phi) 
:= 2\sum_{i=1}^k  \partial_i q \cdot \Big( 
	\langle f_i,\Lap\phi \rangle - \lambda \M( f_i )\Big)+2\sigma^2\sum_{i,j=1}^k \partial^2_{ij}q \cdot \M( f_i f_j)
\end{equ}
where $\partial_i q$ and $\partial^2_{ij}q$ are evaluated at $(\M (f_1) ,\dots,\M( f_k)) $.
\end{Def}
Here ${\mc L}F$ is defined $\mu$- (equivalently, $\nu$-) almost everywhere. Recall that $\phi\mapsto \langle f,\Lap\phi\rangle=\langle \Lap f,\phi\rangle$ is continuous on the 
abstract Wiener space if $f$ is regular enough (e.g. if $f$ is $C^3$).

\begin{Def}
For $F,G\in {\mc C}$ we define a bilinear form  %with domain $\mc D(\mc E)=\mc C$
\begin{equ}[e:def-Dirichlet]
{\mc E}(F,G) :=\int F(\phi) (-{\mc L}G(\phi)) d\nu(\phi) \;.
%=\int \langle DF,DG\rangle_{L^2(dA)}d\nu
\end{equ}
\end{Def}

\begin{Lem} \label{lem:DDform}
We have
\begin{equ}[e:DDform]
{\mc E}(F,G) = \frac12 \int \langle DF(\phi),DG(\phi)\rangle_{L^2(\M)}d\nu(\phi)\;.
\end{equ}
In particular, ${\mc E}$ is symmetric and positive semidefinite on $\mc D({\mc L})^2$.
\end{Lem}

\begin{Rem}\label{rem:invariance}
Taking $F\equiv 1$ in \eqref{e:def-Dirichlet} we have $DF=0$, and by
Lemma~\ref{lem:DDform} one has
  $\int {\mc L}G(\phi)d\nu(\phi)=0$ for any $G\in {\mc C}$.
% Note that this does not necessarily  that the 
\end{Rem}

\begin{Rem}
Lemma~\ref{lem:DG} and Lemma~\ref{lem:DDform} together 
implies that ${\mc E}(F,G)$ defined in \eqref{e:def-Dirichlet} does not depend on the representation 
of $F,G$ in the form \eqref{e:C-form}. 
Moreover, since ${\mc C}$ is dense in $L^2(\nu)$ by Lemma~\ref{lem:dense}, it follows that ${\mc L}F={\mc L}\tilde F$ $\nu$-a.e. if $F=\tilde F$ $\nu$-a.e., i.e. ${\mc L}F$ uniquely depends on $F$ and not on any particular choice of $q,f_1,\dots,f_k$. 
Also, note that ${\mc L} $ is  linear  on the domain $\mc D(\mc L)$. Indeed, for 
\[
F (\phi) = p (\M (f_1) ,\dots,\M( f_k)) 
\quad \mbox{and} \quad
G (\phi) = q (\M (f_{k+1}) ,\dots,\M( f_n)),
\]
a linear combination has the form $(aF+bG )(\phi)= r(\M (f_1) ,\dots,\M( f_n))$
where $\partial^2_{ij}r =0$ unless $\{i,j\}\subset \{1,\cdots,k\}$ or $\{i,j\}\subset \{k+1,\cdots,n\}$;
this together with 
the independence of ${\mc L}F$ on the representation of $F$ implies linearity of  ${\mc L} $.
\end{Rem}

\begin{Rem}
Note that the  $\mc E$ in \eqref{e:DDform}  has a novel form, in the sense that 
the $L^2$ product in \eqref{e:DDform}, as well as the notion of gradient \eqref{e:gradchar}, depend on the GMC measure $\M$.
To compare with the earlier work, for instance  \cite{AlbRock91},
one usually has a {\it fixed } Hilbert space $(H, \langle ,\rangle_H)$ 
and consider forms such as $\frac12 \int \langle  A(\phi)DF(\phi),DG(\phi)\rangle_{H}d\nu(\phi)$ where $A(\phi)$ is some bounded linear operator on $H$.
In our case since $\M$ does not have a density with respect to a fixed measure (such as Lebesgue meaure), our form  $\mc E$  does not fit into the scope of \cite{AlbRock91}.
Our ``tangent spaces'' of 
$\Phi$ do depend on $\phi\in\Phi$ in a nontrivial way (see Eq.~\eqref{e:tangent} for this heuristic).
It also worth noting at this point that a simpler form
$\frac12 \int \langle  DF(\phi),DG(\phi)\rangle_{H}d\nu(\phi)$
with $H=L^2(\TT,d^2 x)$ which is called  a ``classical'' Dirichlet form
in \cite{AlbRock91}
corresponds to the equation studied by  \cite{Garban2018}, which is formally given by (via a simple change of parameters) 
\footnote{ \cite{Garban2018} proved that when $\gamma\in[0,2\sqrt 2-\sqrt 6)$ one can define a local solution for the suitably renormalized  equation, and obtained convergence of the mollified solutions to the limiting solution; when $\gamma\in[2\sqrt 2-\sqrt 6, 2\sqrt 2-2)$, there is still a notion of local solution but with no convergence result.}
\[
\partial_t \phi = \frac{1}{4\pi} \phi - e^{\gamma \phi} + \xi
\]
where $ \xi$ is the space-time white noise with respect to the Euclidean metric.
The framework of  \cite{AlbRock91}  constructs a diffusion w.r.t. this ``classical'' Dirichlet form.
(\cite[Section~7.II.a)]{AlbRock91} focuses on the $P(\Phi)_2$ case but it is remarked that the H{\o}egh-Krohn case on $\R^2$ with a space cutoff can be treated similarly.) The integration-by-parts formula 
required in their setting can be found in \cite{AlbeverioHK74}, as we recorded above in the beginning of Section~\ref{sec:IBP}, which has the same form as our integration by parts formula  but is w.r.t cylindrical test functionals.
\end{Rem}

\begin{proof}[Proof of Lemma~\ref{lem:DDform}]
Letting $F=p(\M(f_1),\cdots,\M(f_m))$ and $G=q(\M(g_1),\cdots,\M(g_n))$, by definition \eqref{e:def-L} of the generator $\mc L$ one has that the right-hand side of \eqref{e:def-Dirichlet} is equal to
\begin{equs} [e:to-use-IBP]
-2\sum_{i=1}^n  \int p\,\partial_i q\cdot    \langle g_i,\Delta \phi\rangle \, & d\nu(\phi)
+2 \lambda \int p\, \sum_{i=1}^n \partial_i q\cdot \M(g_i) d\nu(\phi)
\\
&  -2\sigma^2  \int p\, \sum_{i,j=1}^n \partial_{i}\partial_{j} q\cdot \M(g_i g_j) d\nu(\phi)
\end{equs}
where we omitted the arguments of $p,q$.
We remark that every term here is indeed integrable, 
as shown in the proof of Lemma~\ref{thm:LIBP}.
For each fixed $i\in \{1,\cdots,n\}$,  we apply integration by parts 
(Theorem~\ref{thm:LIBP}) to the  functional  $2p\,\partial_i q$
in the first term of \eqref{e:to-use-IBP}
with Cameron-Martin direction $g_i$  (which is smooth) and get
\begin{equs}
- 2 & \sum_{i=1}^n  \int p\,\partial_i q\cdot    \langle g_i,\Delta \phi\rangle \,  d\nu(\phi)
=
2\sum_{i=1}^n  \int p\,\partial_i q\cdot    \langle \nabla g_i,\nabla \phi\rangle \,  d\nu(\phi)
\\
&=\sum_{i=1}^n \int \Big(\sigma^2 D_{g_i}  (p\,\partial_i q)
 - 2\lambda p\,\partial_i q \, \M(g_i) \Big)d\nu(\phi)
\\
& = \sum_{i=1}^n \int  \Big( 
2 \sigma^2 p \sum_{j=1}^n \partial_i\partial_j q \cdot \M(g_ig_j)
+ 2 \sigma^2  \sum_{j=1}^m \partial_j p \partial_i q \cdot \M(g_i f_j)
 - 2\lambda p\,\partial_i q \, \M(g_i)\Big)d\nu(\phi)
\end{equs}
where in the last step we computed $D_{g_i}$ 
using \eqref{e:DhG} of Lemma~\ref{lem:DG}.
Note that the first and the third terms
in the last line here cancel the second and the third 
terms in \eqref{e:to-use-IBP}.
Therefore  the above calculation shows that the right-hand side of \eqref{e:def-Dirichlet} is equal to
\begin{equ}[e:rhsDDform]
\sum_{i=1}^n \int  \Big( 
 2 \sigma^2  \sum_{j=1}^m \partial_j p \partial_i q \cdot \M(g_i f_j)
\Big)d\nu(\phi) \;.
\end{equ}
This expression, using \eqref{e:DG}, is equal to the right-hand side of \eqref{e:DDform}.
\end{proof}

We will now induce a bilinear form on $\MeaSp$. 
Define a form on $L^2(\MeaSp,\m)$ by
\begin{equ}[e:EXandE]
\EE(F,G) := \mc E( F\circ \mathbf M,G\circ\mathbf M)
\end{equ}
for $F,G\in\CX$. 
It is clearly symmetric and positive semi-definite by Lemma \ref{lem:DDform}.
$\CX$ is dense in $ L^2(\MeaSp,\m)$ by Lemma~\ref{lem:dense}.

 \begin{Lem}
 The form $\EE$  % defined in \eqref{e:def-Dirichlet}
  is closable 
 for every $\sigma< \sigma_{L^1}=2\sqrt\pi$. 
 \end{Lem}
 \begin{proof}
By 
\cite[Eq.~(1.1.3)]{Fukushima2011}, a sufficient condition for 
the symmetric form $\EE$ to be closable is:  for any sequence $F_n \in \CX$ with $\|F_n\|_{L^2(\MeaSp,\m)}\rightarrow 0$ as  $n\to \infty$ one always has
\begin{equs}
\lim_{  n\to \infty} \EE(F_n,G)\to 0, \qquad \forall G\in \CX \;.
 \end{equs}
 Indeed, with these $F_n,G\in\CX$, denoting $\tilde F_n=F\circ{\mathbf M}$, $\tilde G=G\circ{\mathbf M}$, with $\tilde F_n,\tilde  G \in {\mc C}$, we have:
 \[
|\EE(F_n,G)|=\Big|\int \tilde F_n (-{\mc L}\tilde G) d\nu\Big|
\le \|F_n\|_{L^2(\MeaSp,\m)}  \|{\mc L}\tilde G\|_{L^2(\nu)} \;.
\]
Recall that $\CX\subset L^p(\m)$ for all $p<\infty$.  By the expression of ${\mc L}\tilde G$ \eqref{e:def-L} and Lemma \ref{Lem:sigmafin}, it follows that ${\mc L}\tilde G$ is in $L^p(\nu)$ for all $p<\infty$, which concludes.
\end{proof}

We also denote by $\EE$ the smallest closed extension.

\begin{Prop}\label{prop:regular}
 $\EE$ is a Dirichlet form which is regular on $L^2(\MeaSp,\m)$.
\end{Prop}
\begin{proof}
Recall from \cite[Section~1.1]{Fukushima2011} (or beginning of Section~\ref{sec:solution-Dirichlet}) that for  $\EE$ to be regular we need to prove that $\EE$ possesses a core.
For this we need that
$\CX$ is dense in $C_0(\MeaSp)$ - the space of compactly supported continuous functions on $\MeaSp$ with uniform norm.
This is the content of  Lemma~\ref{lem:dense}.

Clearly, it is also a standard core, namely $\CX$ is a dense linear subspace of $C_0(\MeaSp)$;
and for any $\eps>0$, a cutoff function $\phi_\eps(t)$ with 

1) $\phi_\eps(t)=t$ for $t\in[0,1]$ 

2) $\phi_\eps(t)\in [-\eps,1+\eps]$ for all $t\in \R$ and 

3) $\phi_\eps(t')-\phi_\eps(t)\in [0,t'-t]$ for $t<t'$,

and $F(A)=q(A(f_1),\dots,A(f_k))$,
 we have that $\phi_\eps(F)\in \CX$ since $\phi_\eps\circ q$ satisfies the requirements in Definition~\ref{def:mcC}.

We  also need to check that $\EE$ is Markovian. Indeed, taking a  cutoff function $\phi_\eps$ as above which is further assumed to be differentiable,
one has that for each $F(A)$ as above
\begin{equ}
D (\phi_\eps\circ F) (A)= 2 \sum_{i=1}^k   (\phi_\eps'\circ q)   (A( f_1 ),\dots, A( f_k)) \cdot  \partial_i q (A( f_1 ),\dots, A( f_k))
 \cdot f_i 
\end{equ}
so that (using the calculation \eqref{e:rhsDDform} and the equivalence \eqref{e:EXandE})
\[
\EE(\phi_\eps\circ F,\phi_\eps\circ F) =  \frac12  \int  (\phi_\eps'\circ q)^2  (A( f_1 ),\dots, A( f_k))  \langle DF(\phi),DG(\phi)\rangle_{L^2(\M)}  d\nu(\phi)\;.
\]
Since $ \phi_\eps' \in (0,1]$ and $ \langle DF(\phi),DG(\phi)\rangle_{L^2(\M)}  \ge 0$, one has 
$\EE(\phi_\eps\circ F,\phi_\eps\circ F)\le \EE(F,F)$,
namely $\EE$ is Markovian.
\end{proof}

\subsection{Existence of diffusion process}
\label{sec:diffusion}

\begin{Prop}\label{prop:existsA}
There exists a unique $\m$-symmetric diffusion  $\HuntA = (\Omega, \mathcal F, (A_t), (P_z))$  on $\MeaSp$ associated to $\EE$.
\end{Prop}

Here the diffusion is called $\m$-symmetric if its associated semigroup $T_t$ satisfies 
$\int f T_t g \,d\m =\int g T_t f \,d\m $ for any non-negative measurable functions $f,g$ and $t>0$.
If one further has $T_t 1=1$ for any $t>0$, then taking $f=1$ one has
 $\int  T_t g \,d\m =\int g \,d\m $ so the measure is invariant.

Here the uniqueness of $\m$-symmetric Hunt process is up to equivalence in the sense of \cite[Section~4.2]{Fukushima2011}. Recall from \cite[Section~4.5]{Fukushima2011}
that a Hunt process  is called a diffusion if 
\begin{equ}[e:conti-path]
P_z(t\mapsto A_t\mathrm{\ is\ defined\ and\ continuous\ for\ all\ }t\in(0,\zeta))=1
\end{equ}
for every $z\in \MeaSp$, where $\zeta$ is the lifetime of $\HuntA$ in the sense of \cite[Appendix~A.2]{Fukushima2011}; also recall that by  \cite[Theorem~4.5.1]{Fukushima2011},
there exists an $\m$-symmetric diffusion on $\MeaSp$ which is equivalent with $\HuntA$ if and only if 
$\HuntA$ is of continuous paths for quasi-every starting point, i.e. there exists a properly exceptional set $N$ such that \eqref{e:conti-path} holds for every $z\in \MeaSp\backslash N$.
Here $N$ being properly exceptional set  means that $N$ is nearly Borel measurable (see \cite[Appendix~A.2]{Fukushima2011}), $\m(N)=0$ and $\MeaSp\backslash N$ is $\HuntA$-invariant, see \cite[Section~4.1]{Fukushima2011}.

\begin{proof}
Since $\EE$ is regular, by
\cite[Theorem~7.2.1]{Fukushima2011}, there exists an $\m$-symmetric Hunt process  associated to $\EE$. By \cite[Theorem~7.2.2 or Theorem~4.5.1]{Fukushima2011}, for this Hunt process to be a diffusion we need to show locality of $\EE$.

To prove locality, let $F,G$ in $\CX$ with disjoint (compact) support. We want to show: $\EE(F,G)=0$. Take $(e_k)$ a sequence of smooth functions dense in $C_0(\TT)$ and let 
$$
d_n(A,A')=\sum_{k=0}^n 2^{-k}\left(\big|  A( e_k) -A'(e_k) \big|\wedge 1\right)
$$
and $d=\lim d_n$; then $d$ metrizes $\MeaSp$. Then 
$$\inf\{d(A,A'):A\in Supp(F), A'\in Supp(G)\}>0$$ hence there is $n$ such that 
$$\inf\{d_n(A,A'):A\in Supp(F), A'\in Supp(G)\}>0$$
i.e. the images of $Supp(F)$ and $Supp(G)$ under $A\mapsto L_n(A):=(A(e_0),\dots,A(e_n))$ are disjoint compact sets in $\R^{n+1}$. We can find
 $f,g\in C^\infty_c(\R^{n+1})$ which have disjoint supports, \footnote{Note that the function $f$ here shouldn't be confused with the notation $f_0,\cdots,f_k$; they are different objects.}
such that $f=1$ on $L_n(F)$ and $g=1$ on $L_n(G)$,
and thus write
\[
F=F\times (f\circ L_n),
\quad G=G\times (g\circ L_n)
\]
where $\times$ is just pointwise multiplication.
Therefore for $F=q(A(f_0),\cdots,A(f_k)) $
\begin{equs}
D F = 2 & \sum_{i\le k}\partial_{i} (qf) (A( e_0 ),\cdots, A( f_k), A(e_0),\cdots,A(e_n))\,f_i 
\\
 &+ 2 \sum_{i>k}\partial_i (qf) (A( e_0 ),\cdots, A( f_k), A(e_0),\cdots,A(e_n)) \,e_i 
\end{equs}
and similarly for $DG$.
 By direct inspection of \eqref{e:DDform} and \eqref{e:DG}, 
 and the fact that $\partial^\alpha f \partial^\beta g =0$ for any $\alpha,\beta\in\{0,1\}$,
 it follows that $\EE(F,G)=0$, i.e. $\EE$ is local (and so is its extension, see Theorem 3.1.2 in \cite{Fukushima2011}).

We also check strong locality. Let $F\in\CX$ with compact support $K\subset \MeaSp$, and $G\in\CX$ which is a constant say $\overline G\in \R$ on a neighborhood $U$ of $K$. Since the topology of $\MeaSp
$ is generated by the maps $A\mapsto A( f_n)$, there is $n\geq 0$ and a neighborhood $V$ of $L_n(K)$ such that $V\subset L_n(U)$. So 
\[G=g\circ L_n+(G-g\circ L_n)\] 
where $g$ is constant $\overline G$ on $V$, and the second summand vanishes on a neighborhood of $K$. By the previous argument and again by direct inspection of \eqref{e:DDform}, \eqref{e:DG}, it follows that $\EE(F,G)=0$, i.e. $\EE$ is strongly local (and so is its extension, see Exercise 3.1.1 in \cite{Fukushima2011}). Strong locality expresses the absence of killing, see Theorem 4.5.3 in \cite{Fukushima2011}.

A similar argument shows that $\CX$ is a {\em special} standard core (see I.1 in \cite{Fukushima2011}).
Namely, for any compact set $K\subset \MeaSp$ and a relatively compact open set $U$ with $K\subset U$, one can construct an element $F\in \CX$ such that $F\ge 0$, $F=1$ on $K$ and $F=0$ on $\MeaSp\backslash U$ by pulling back such a function on $\R^{n+1}$ using the map $L_n$.
\end{proof}

\subsection{Fukushima decomposition and weak solution}
\label{sec:Fukushima}
%We come back to our original problem,  
%namely to give a meaning to a notion of weak solution to %\eqref{e:SRicci} or
% \eqref{e:Ricci-A}.

We prove Theorem~\ref{thm:main} in this subsection. 
%In particular this gives us the existence of a weak solution to   \eqref{e:Ricci-A}.\julien{not sure about 'in particular' if we define existence of weak solution by the statement of the theorem}

\begin{Rem}
Formulating weak solution to be the process having one-dimensional projections given by 
\eqref{e:1d-proj-theo} is typical and analogous formulations exist for other  SPDEs, see e.g. \cite[(0.8)]{AlbRock91}.
This is the usual formulation even for some linear SPDEs which do no require Dirichlet forms to provide solutions; for instance, for the stochastic heat equation with multiplicative noise (i.e. Cole-Hopf transformed KPZ)
$\partial_t\phi  =\Delta\phi + \phi \xi_0$, a weak solution with initial condition $\phi_0$
 is defined by a probability space and filtration together with 
a process $\phi$ such that 
$M_f (t) := \langle\phi_t,f\rangle - \langle\phi_0,f\rangle - \int_0^t \langle\phi_s,\Delta f\rangle ds$ is a martingale for any smooth test function $f$,
and the quadratic variation of $M_f$ is given by $\int_0^t \langle\phi^2, f^2\rangle ds$, see \cite[Def.4.10]{BG} or \cite[Section~4.3]{Corwin2020BAMS}; here $\langle,\rangle$ denotes the $L^2$ inner product. In other words (by L\'evy characterization)
$\phi(f) $ satisfies  
$d\langle \phi,f\rangle = \langle\phi,\Delta f\rangle dt + \langle\phi^2,f^2\rangle^{1/2} d\beta_t$. 
\end{Rem}

% \hao{Since weak solution is now defined in Section 1, I reworded the following two paragraphs. (Before we defined weak solution here after the Ito derivation, which was not  good.) Maybe you could have one more look if the logic is correct}
 
Recall that we eventually want to have the 1-dimensional projection to satisfy \eqref{e:1d-proj} or \eqref{e:1d-proj-theo},
namely
 \begin{equ}
d\int_{\TT}  f \, \area_g
=2\Big(\int_{\TT}   f\Delta\phi \area_0
-\lambda\int_{\TT}  f\, \area_g\Big)dt
+2\sigma \int_{\TT} f e^\phi \xi_0 \area_0 \;.
%+2\sigma\left(\int_{\TT}  f^2\area_g\right)^{\frac12}d\beta_t \;.
\end{equ}
(In \eqref{e:1d-proj} we identified the martingale term $\int f e^\phi \xi_0 \area_0$ as $\left(\int  f^2\area_g\right)^{\frac12}d\beta_t^f$. For two test functions $f_i,f_j$ the time derivative of the covariation
of  $\int f_i e^\phi \xi_0 \area_0$ and  $\int f_j e^\phi \xi_0 \area_0$ is $\int f_if_j \area_g$.)

For $F = q(\int f_0 \area_g,\dots,\int f_k \area_g) $ we can use
this 1-dimensional projection,
 It\^o's formula and definition of ${\mc L}$ in \eqref{e:def-L} to {\it formally} derive:
\begin{equs}
\frac{dF}{dt} &=\sum_i 2\partial_i q \cdot
\Big(
\int f_i\Lap\phi \area_0  -\lambda\int f_i \area_g
+ \sigma \int f_i e^\phi \xi_0 \area_0
\Big)
+2\sigma^2\sum_{i,j}\partial^2_{ij}q\int f_if_j \area_g
\\
&= {\mc L}F(\phi) 
+  2 \sigma \sum_i  \partial_i q 
 \int f_i e^\phi \xi_0 \area_0 \;.
\end{equs}
In view of this, 
if an $\MeaSp$-valued  process $A_t$ 
%constructed in Proposition~\ref{prop:existsA} 
is  a  weak solution to \eqref{e:Ricci-A},
then for any
$F (A)= q(A( f_0 ),\dots,A( f_k )) \in \CX$
we have that $M_t: = F(A_t)-F(A_0) - \int_0^t \mc L F(A_s)ds$ is a martingale
whose quadratic variation is 
\begin{equ}[e:QVM]
\langle M\rangle_t 
 = 4\sigma^2 \sum_{i,j=0}^k \int_0^t \partial_i q \, \partial_j q \cdot  A_s( f_i f_j)\,ds  \;.
\end{equ}
With the diffusion $\HuntA = (\Omega, \mathcal F, (A_t), (P_z))$ obtained above in Proposition~\ref{prop:existsA}, we will prove that this $A_t$ is  a weak solution.
 
Recall that for such an $F\in\CX$, the gradient is given by
 \begin{equ}[e:gradCX]
DF(A)=\sum_{i=1}^k\partial_i q(A(f_0 ),\dots, A(f_k)) f_i
\end{equ}
so that $DF\in\CX\otimes C^\infty(\TT)$.

Below we write AF for ``additive functional''.

Let $F\in\CX$, $(A_t)_{t\ge 0}$ the process in $\MeaSp$ associated to ${\mc E}$. We consider the continuous  
AF  (\cite[Section~5.2]{Fukushima2011})
\[
Y^{[F]}_t=F(A_t)-F(A_0)\;.
\]
By  \cite[Theorem~5.2.2]{Fukushima2011}, $Y^{[F]}$ admits a unique Fukushima decomposition 
\begin{equ}[e:F-decom]
Y^{[F]}=M^{[F]}+N^{[F]}
\end{equ}
 where $M^{[F]}$ is a martingale AF of finite energy 
 and $N^{[F]}$ is a zero-energy continuous AF. 

Namely, $M^{[F]}$ is a finite c\`adl\`ag AF such that for each $t>0$, 
$\E_z(M_t^2)<\infty$ and $\E_z (M_t)=0$ for quasi-every $z\in \MeaSp$ where $\E_z$ is the expectation for the measure $P_z$, with energy 
\[
e(M^{[F]}):=\lim_{t\to 0} \frac{1}{2t} 
\E_\m[(M^{[F]}_t)^2]
\]
 being finite; and  $N^{[F]}$ is a finite continuous AF, with $e(N^{[F]})=0$ and $\E_z[|N_t^{[F]}|]<\infty$ for quasi-every  $z\in \MeaSp$  for each $t>0$.

In particular $M^{[F]}$ admits a quadratic variation $\langle M^{[F]} \rangle$ which is a positive continuous AF such that $\E_z[\langle M^{[F]} \rangle_t] = \E_z [(M^{[F]}_t)^2]$ for quasi-every  $z\in \MeaSp$ and $t>0$. For the quadratic variation $\langle M^{[F]} \rangle$ we have the following lemma.

\begin{Lem}\label{lem:Quad-M}
Let $F\in \CX$ and $M^{[F]}$ be the martingale AF in \eqref{e:F-decom}. We have
\begin{equ}[e:quad-M]
\langle M^{[F]} \rangle_t=\sigma^2\int_0^t \|DF(A_s)\|^2_{L^2(A_s)}ds \;.
\end{equ}
\end{Lem}

\begin{proof}
The quadratic variation $\langle M^{[F]}\rangle$ is a positive AF, which is associated with a Revuz measure  $\mu_{\langle M\rangle}$ via the Revuz correspondence (\cite[Section~5]{Fukushima2011}).
By  \cite[Theorem~5.2.3]{Fukushima2011}, this Revuz measure has $\mu_{\langle M\rangle}(G)=2\EE(F\cdot G,F)-\EE(F^2,G)$.    
From \eqref{e:gradCX}
we readily check that the Leibniz rule $D(FG)=F\cdot DG+G\cdot DF$ holds, which implies
\begin{equ}[e:Revuz]
d\mu_{\langle M\rangle}(A)=\sigma^2\|DF(A)\|^2_{L^2(A)}d\m (A) \;.
\end{equ}
Remark that with $F = q(\int f_1 dA,\dots,\int f_k dA) $, the functional
\begin{equs}
A\mapsto  & \|DF(A)\|^2_{L^2(A)} \\
&=\sum_{i,j=1}^k \partial_i q\big( A(f_1),\dots,A(f_k) \big)\partial_j q\big(A(f_1),\dots,A(f_k)\big) 
A( f_if_j)
\end{equs}
is also in $\CX$, and in particular continuous on $\MeaSp$. It is then standard (for instance following the same lines as the proof of \cite[Proposition~4.5]{AlbRock91}) to show that the Revuz measure corresponding to the right-hand side of \eqref{e:quad-M} is also \eqref{e:Revuz}, thus the lemma follows.
\end{proof}

To identify the diffusion as the weak solution we have the following  more concrete representations (as required by \eqref{e:QVM}).

\begin{Lem}\label{lem:identify-M}
For $F\in \CX$ with $F (A)= q(A(f_1),\cdots,A(f_k))$ we have
\begin{equ}[e:MF-brac]
\langle M^{[F]} \rangle_t=  4\sigma^2 \int_0^t
	 \sum_{i,j} \partial_i q \partial_j q \cdot A_s(f_i f_j )\,ds \;.
\end{equ}
In particular, for $F(A)=A(f)$ one has
\begin{equ}[e:MFAf]
\langle M^{[F]} \rangle_t=   4\sigma^2  \int_0^t A_s(f^2)\,ds \;,
\end{equ}
and $M^{[F]}_t =2\sigma \int_0^t  (A_s(f^2))^{\frac12}\,d\beta^f_s$ for a  one-dimensional Brownian motion $\beta^f$ as required in \eqref{e:1d-proj}.

Moreover, for $F=A(f)$ and $G=A(g)$ one has
\begin{equ}[e:Covariation]
\langle M^{[F]},M^{[G]} \rangle_t=   4\sigma^2  \int_0^t A_s(fg)\,ds \;.
\end{equ}
\end{Lem}

\begin{proof}
By the calculation of $DF$ from \eqref{e:gradCX} and Lemma~\ref{lem:Quad-M}
the claims \eqref{e:MF-brac} and \eqref{e:MFAf}   follow. Remark that $F(A)=A(f)$ is not in $\CX$ (not compactly supported); one obtains the desired result by standard truncation/localization arguments. The statement 
on identification of $M^{[F]}$
follows from the continuity of the AF $Y^{[F]}$ which implies continuity of $M^{[F]}$ together with martingale representation theorem.
\end{proof}

\begin{Lem}\label{lem:identify-N}
Let $F\in \CX$ and $N^{[F]}$ be the zero energy continuous AF in \eqref{e:F-decom}. We have
\[
N^{[F]}_t=\int_0^t  {\mc L}F(A_s)ds \;.
\]
In particular, for $F=A(f)$ one has
$N^{[F]}_t=2 \int_0^t   \Big(  \area_0( f\Delta\phi_s )
-\lambda    A_s ( f )\Big)ds$ with $\phi_s=\mathbf M^{-1} A_s$.
\end{Lem}

\begin{proof}
 We have by integration by parts:
$${\mc E}(F,G)=\int Gd\nu_F$$
where $d\nu_F(A)={\mc L}F(A)d\m(A)$; here $\nu_F$ has a locally integrable density with respect to $\m$ (actually integrable, see Lemma \ref{Lem:sigmafin}). By \cite[Corollary~5.4.1]{Fukushima2011}, $N$ is an AF with Revuz measure $\nu_F$ (i.e. $N=N^+-N^-$ where $N^\pm$ is a positive AF with Revuz measure $\nu_F^\pm$).

Since ${\mc L}F$ is measurable on $\MeaSp$ locally compact, it can be approximated in $L^p(\m)$ by continuous functions. $t\mapsto {\mc L}F(A_t)$ is an $L^p$ limit of continuous adapted processes, hence is progressively measurable. In particular $t\mapsto \int_0^t{\mc L}F(A_s)ds$ is well defined as a process and one can identify (arguing as in Example 5.1.1 of \cite{Fukushima2011}).
$$N_t=\int_0^t{\mc L}F(A_s)ds\;.$$
\end{proof}
In particular remark that for $F\in\CX$, $N^{[F]}$ has bounded variation, so that $Y^{[F]}$ is a continuous semimartingale, and \eqref{e:F-decom} is simply its semimartingale decomposition.

\begin{proof}[Proof of Theorem~\ref{thm:main}]
The claim of the theorem now immediately follows from Lemma~\ref{lem:identify-M} and Lemma~\ref{lem:identify-N}.
\end{proof}

\paragraph{Absorption.}
Let  $A_t(1)=\int_{\TT} A_t(dx)$ be  the total volume of the torus $\TT$.
By applying these results 
we have:
\begin{Cor}\label{cor:absorb}
 The process $(A_t(1))_{t\geq 0}$ is a.s. absorbed at 0 in finite time. 
\end{Cor}
\begin{proof}
Consider the function $f\equiv 1$; one then obtains a simple autonomous SDE satisfied by $A_t(1)$ 
\begin{align*}
dA_t(1)&=2\left(\int \phi_t(\Delta 1) dA_0-\lambda A_t(1)\right)dt+2\sigma\sqrt{A_t(1)}d\beta_t\\
&=2\sigma\sqrt{A_t(1)}d\beta_t -2\lambda A_t(1)dt
\end{align*}
where $\beta$ is a 1-dimensional Brownian motion.
One recognizes the evolution of a continuous-state branching process (CSBP), which is also continuous in time (see e.g. \cite{lambert2007quasi,lamperti1967continuous}). A solution for $\lambda\geq 0$ is stochastically dominated by a solution for $\lambda=0$. Setting $\lambda=0$, one further recognizes the SDE satisfied by a square Bessel process of dimension 0 (BESQ(0), also known as the Feller diffusion). That process is a.s. absorbed at 0 in finite time. 
\end{proof}
Consequently,  recalling that we have defined
$\MeaSp:={\mc M}(\TT)\setminus\{0\}$,
the lifetime $\zeta$ in \eqref{e:conti-path} is a.e. finite; 
the process evolves continuously in $\MeaSp$ until it is absorbed at 0, which occurs in finite time. Alternatively we can consider ${\mc M}(\TT)$ itself as the state space, and then $0$ is an absorbing state.

\section{Extensions}
\label{sec:extensions}

In the case of the torus, we started from a heuristic derivation of the generator and arrived at the Dirichlet form \eqref{e:DDform}. Since the formal dynamic of the SRF is defined in terms of intrinsic, local quantities, it is natural to expect it can be defined on surfaces of general topology. However, it will appear momentarily that the use of a reference {\em flat} metric on the torus played a subtle role in these heuristics, and in the general case one needs to properly account for the curvature of the reference metric, in order to obtain processes defined invariantly. For that purpose it will be convenient to reverse the logic leading from \eqref{e:def-Dirichlet} to \eqref{e:DDform} and use the natural generalization of \eqref{e:DDform} as a starting point instead.

In Section \ref{sec:surfaces}, we address the case of general compact surfaces. The key (and only substantial) difference with the case of the torus lies in the use of reference metric with curvature. We explain the needed modifications and the rest of the argument follows as in the torus case {\em mutatis mutandis} (in particular, the implementation of Dirichlet forms in Section \ref{sec:solution-Dirichlet} is unchanged). In Section \ref{Sec:insert}, we explain the additional arguments needed to accommodate vertex insertions.

\subsection{Compact surfaces}
\label{sec:surfaces}

Considering the SRF on more general surfaces involves two key difficulties: curvature of the reference metric, and the potential presence of boundary components. For clarity we explain separately the arguments involved in incorporating these two elements.

\paragraph{Closed surfaces.}
We turn to consider a general closed Riemannian surface $\Sigma$ with a reference metric $g_0$.
According to the discussion  in Section~\ref{sec:Intro} (in particular below \eqref{e:6piV-S}), it would seem to be natural to try to construct a SRF of the form
\begin{equ}\label{e:guess}
\partial_t\phi\stackrel{??}{=}-K_g-\lambda+\sigma\xi_g
=e^{-2\phi} (\Lap\phi-K_0)-\lambda+\sigma\xi_g \;,
\end{equ}
so that for $A=\Wick{e^{2\phi}\area_0}$  we expect the following one dimensional projection (compare with \eqref{e:1d-proj-theo})
\begin{equ}\label{Eq:guess}
d   A_t(f)
\stackrel{??}{=}2\Big(  \area_0( f\Delta\phi_t )
-  \area_0( f K_0 )
-\lambda    A_t ( f )\Big)dt
+2\sigma\left( A_t(f^2)\right)^{\frac12}d\beta_t^f \;.
\end{equ}
With this `naive' guess, 
 for $F=q(A(f_1),\dots,A(f_n))$ we have (compare with \eqref{e:def-L})
\[
{\mc L}F(\phi)\stackrel{??}{=}
\sum_i 2\partial_i q \cdot \Big(\area_0( f_i\Lap\phi )- \area_0( f_iK_0)-\lambda A( f_i )\Big)+2\sigma^2\sum_{i,j}\partial^2_{ij}q  \cdot A( f_if_j)\;.
\]
However this {\it formal} derivation -- as we have put questions marks on the above identities -- turns out to fail to correctly account for the ``quantum" correction in Liouville CFT, see Remark~\ref{rem:Qdyn}; we will explain about this momentarily.

Here we follow closely \cite{guillarmou2016polyakov}, which we refer to for a detailed treatment. 
Recall the convention comparison $ (\phi,\lambda,\sigma) \leftrightarrow   (X,\mu,\gamma)$
of \eqref{e:lambda-sigma-mu}, and in this section we turn to the ``probabilists' convention'' for easier 
referencing to  \cite{guillarmou2016polyakov}. With a surface $\Sigma$ and a reference metric $g_0$ we consider the action (see  \cite[(2.2)]{guillarmou2016polyakov})  \footnote{Note that \cite{guillarmou2016polyakov} uses the letter $K$ for scalar curvature which equals twice the Gauss curvature, whereas we use the letter $K$ for Gauss curvature; and \cite{guillarmou2016polyakov} uses the letter $\varphi$ for the field instead of $X$ as here.}
\begin{equation}\label{eq:Liouvact}
S_L(g_0,X)=\frac 1{4\pi}\int_\Sigma (|\nabla X|^2+2QK_0X+4\pi\mu e^{\gamma X})\, \area_0
\end{equation}
where 
\begin{equation}\label{eq:Qdef}
Q=\frac 2\gamma+\frac\gamma 2\;.
\end{equation}
This induces the measure on fields (see  \cite[(3.1)]{guillarmou2016polyakov})
$$d\nu_{g_0}(X)=\exp\left(-S_L(g_0,X)\right){\mc D}X
=e^{-\frac 1{4\pi}\int_\Sigma (2QK_0 X+4\pi\mu e^{\gamma X})\,\area_0}\left(e^{-\frac 1{4\pi}\int_\Sigma |\nabla X|^2 \,\area_0}{\mc D}X\right)
$$
where ${\mc D}X$ is the formal flat measure on fields. The last bracket is interpreted as a $\sigma$-finite measure on paths as in \eqref{eq:pathint}, viz. as the sum of a zero-mean GFF and a ``Lebesgue-distributed" constant; the rest of the action is then an almost everywhere defined Radon-Nikod\'ym derivative. The quadratic part is characterized, up to multiplicative constant, by the Cameron-Martin formula:
\begin{equation}\label{eq:CMgen}
e^{-\frac 1{4\pi}\int |\nabla (X+h)|^2 \,\area_0}{\mc D}(X+h)
=e^{-\frac 1{2\pi}\int \nabla X\cdot\nabla h \,\area_0-\frac 1{4\pi}\int |\nabla h|^2 \,\area_0}
\left(e^{-\frac 1{4\pi}\int |\nabla X|^2\area_0}{\mc D}X\right)
\end{equation}
for $h\in H^1(\Sigma)$, as in Lemma \ref{lem:GIBP} (in particular \eqref{e:to-be-diff-t}).

Indeed, remark that, from the local nature of the Liouville measure, there is no difficulty in constructing it on surfaces along the following lines: cover $\Sigma$ by finitely many complex disks $(U_i)_{1\leq i\leq n}$; for any $V\subset\subset U_i$, the restriction of the zero-mean GFF to $V$ is absolutely continuous w.r.t. to the zero-boundary GFF on $U_i$. Then one can define the Liouville measure on each $U_i$ and patch them together. Moreover one can assume that each $U_i$ carries isothermal coordinates. Alternatively one can retrace the steps of the planar construction and check that it carries over to surfaces. One also has the basic moment estimates used in Lemma \ref{Lem:sigmafin}.

Let $\hat g_0=e^{2\psi_0}g_0$ be another reference metric with volume form $\hat\area_0$. Then the GMC regularization introduces the following anomalous scaling 
(see  \cite[(3.12)]{guillarmou2016polyakov}):
\begin{equation}\label{eq:anomscal}
\Wick{e^{\gamma X} \hat \area_0}_{\hat g_0}
=e^{(2+\gamma^2/2)\psi_0}\Wick{e^{\gamma X} \area_0}_{g_0}
\end{equation}
i.e. $\Wick{e^{\gamma \hat X} \hat\area_0}_{\hat g_0}=\Wick{e^{\gamma X}\area_0}_{g_0}$ if $\hat X=X-Q\psi_0$ (here $:\;:_{g_0}$ denotes the limit of an $\eps$-regularization scheme such as \eqref{eq:epsreg} {\em where $\eps$ is measured in $g_0$}).
Moreover,
we have the conformal anomaly (see \cite[Proposition 4.2]{guillarmou2016polyakov})
\begin{equ}[e:confanom]
\int F(X)d\nu_{\hat g_0}(X)\propto\int F(X-Q\psi_0)d\nu_{g_0}(X)
\end{equ}
and consequently the pushforward of $\nu_{g_0}$ by 
$X\mapsto  M_X \stackrel{def}{=} \Wick{e^{\gamma X}\area_0}_{g_0}$ (a $\sigma$-finite measure on $\MeaSp:={\mc M}(\Sigma)\setminus\{0\}$) does not depend (up to multiplicative constant) on the choice of reference metric $g_0$, just on the Riemann surface structure of $\Sigma$. This justifies the choice of \eqref{eq:Qdef} in the definition of the action \eqref{eq:Liouvact}. Let us denote that measure by $\m$.

We have the integration-by-parts formula:
 \begin{equ}[e:IBPsurface]
 4\pi\int D_hF(X)d\nu_{g_0}(X)
 =\int F(X)\left(2 \area_0 \Big( \nabla_0 h\cdot\nabla_0X + Q K_0h \Big)+4\pi\mu\gamma M_X(h)\right)d\nu_{g_0}(X)
 \end{equ}
derived from \eqref{eq:CMgen} along the same lines as Theorem \ref{thm:LIBP}.

Similarly to Definition~\ref{def:mcC}   %\eqref{e:zero-th-co}
 and Lemma~\ref{lem:DG}, %\eqref{e:DhG}, 
 one can define test functionals and evaluate their Fr\'echet derivatives as follows: if $M_X=\Wick{e^{\gamma X} \area_0}_{g_0}$ and $F$ is a test functional of the form
$$F(X)=q(M_X (f_0),\dots, M_X( f_k))$$
where the $f_i$'s are smooth on $\Sigma$ and $q:\R^{k+1}\rightarrow\R$ is $C^2$, 
the Fr\'echet derivative of $F$ in the smooth direction $h$ is
\begin{equ}[e:DhFX]
D_h F(X)= \gamma \sum_i \partial_i q  (M_X (f_0),\dots, M_X( f_k)) \,M_X ( h f_i )\;.
\end{equ}
Similarly to \eqref{e:DG}, the gradient $DF$ of $F$ {\em w.r.t. to the Liouville $L^2$ norm} is thus given by
\begin{equ}[e:DFX]
DF(X)=\gamma\sum_i\partial_i q  (M_X (f_0),\dots, M_X( f_k)) f_i \;.
\end{equ}

From the study of the torus case in Section~\ref{sec:Closability}, it is at this stage natural to take as starting point the following Dirichlet form:
$${\mc E}(F,F)=\int \|DF\|^2_{L^2(M_X)}d\nu_{g_0}$$
which depends on the choice of reference metric $g_0$ only through a multiplicative constant. Running the computation of Lemma \ref{lem:DDform} in reverse order, we have (write $\nu=\nu_{g_0}$)
\begin{align*}
{\mc E}(F,F)
&=\int\|DF(X)\|^2_{L^2(M_X)}d\nu(X)\\
&\stackrel{\eqref{e:DFX}}{=}
	\gamma^2\int \sum_{i,j}\partial_i q \,\partial_j q \,M_X( f_if_j ) \, d\nu(X)\\
&\stackrel{\eqref{e:DhFX}}{=}
	\int \Big(\gamma\sum_i D_{f_i}( q \partial_i q)-\gamma^2\sum_{i,j}q \partial^2_{ij}q M_X( f_if_j ) \Big) d\nu(X)\\
&\stackrel{\eqref{e:IBPsurface}}{=}
	-\int q \Big(\sum_i \partial_i q \cdot \Big(\frac\gamma{2\pi} \area_0( f_i\Lap X)-\frac{Q\gamma}{2\pi} \area_0(f_i K_0)-\mu\gamma^2 M_X( f_i)\Big) \\
&\hphantom{oooooooooo}+\gamma^2\sum_{i,j}\partial^2_{ij}q M_X( f_if_j)\Big)d\nu(X)\\
&=\int F(-{\mc L}F)d\nu(X)
\end{align*}
where 
\begin{equation}\label{e:gensurf}
{\mc L}F=\gamma^2\sum_{i,j}\partial^2_{ij}q  M_X( f_if_j )
+\sum_i \partial_i q \left(\frac\gamma{2\pi}\area_0( f_i\Lap X ) -\frac{Q\gamma}{2\pi}\area_0( f_i K_0)-\mu\gamma^2 M_X ( f_i)\right)\;.
\end{equation}
More generally, we have 
$${\mc E}(F,G)=\int\langle DF(X),DG(X)\rangle_{L^2(M_X)}d\nu(X)=\int F(-{\mc L}G)d\nu(X)\;.$$

This corresponds to the formal dynamics % (compare with \eqref{e:Ricci-A})
\begin{equation}\label{e:Ricci-A-surf}
\partial_t A_t=\frac{\gamma}{2\pi}\Lap X \area_0-\frac{Q\gamma}{2\pi}K_0\area_0
-\mu\gamma^2 A_t+\gamma\sqrt 2\xi_g A_t
\end{equation}
or the 1d dynamics % (compare with \eqref{e:1d-proj})
\begin{equation}\label{e:Ricci-Af-surf}
dA_t (f)
=\left(\frac{\gamma}{2\pi} \area_0( f\Lap X ) -\frac{Q\gamma}{2\pi}\area_0( fK_0 )
 -\mu\gamma^2 A_t (f)\right)dt +\sqrt 2 \gamma \, \sqrt{A_t( f^2)} d\beta^f_t
\end{equation}
for $f$ a smooth function on $\Sigma$.

\begin{Rem}\label{rem:Qdyn}
Up to replacing $\gamma X$ (LQFT convention) with $2\phi$ (earlier convention),
  matching parameters
 by \eqref{e:lambda-sigma-mu}, 
 and a time change $t\mapsto 2\pi t$, the equation \eqref{e:Ricci-Af-surf} can be written as
 \[
 dA_t (f)
= 2 \left( \area_0( f\Lap\phi ) -\left(\frac{Q\gamma}{2}\right) \cdot \area_0( fK_0 )
 - \lambda A_t (f)\right)dt +2\sigma \, \sqrt{A_t( f^2)} d\beta^f_t
 \]
 which 
  differs from the na\"ive guess \eqref{Eq:guess} by a factor $Q\gamma/2$ (in front of the curvature term),
 which would be equal to $1$ in the ``classical'' case $Q=2/\gamma$, but is actually equal to  $1+\gamma^2/4 = 1+\sigma^2/(4\pi)$
 for the ``quantum'' value of $Q$ \eqref{eq:Qdef}. 
  
The right dynamic for the conformal factor should be (compare with the na\"ive guess \eqref{e:guess}):
\[
\partial_t\phi   
=e^{-2\phi} \Big(\Lap\phi- \Big(\frac{Q\gamma}{2}\Big)K_0\Big)-\lambda+\sigma\xi_g\;.
\]
 The discrepancy can explained as follows: the formal derivation of \eqref{Eq:guess} interprets the Liouville measure as a 2-form, but it actually transforms as a $Q\gamma$-form (see \eqref{eq:anomscal}).
\end{Rem}

The Fukushima decomposition follows in the same way as Section~\ref{sec:Fukushima}.

\paragraph{Bordered surfaces.}

Here we consider the case of bordered surfaces: $\Sigma$ is an open surface, and its boundary $\partial\Sigma$ consists of finitely many (possibly zero) circles. For simplicity (to avoid introducing geodesic curvature), we fix a reference metric $g_0$ such that the neighborhood of each boundary component is isometric to a flat half-cylinder $(\R/\ell\Z)\times [0,\eps)$ ($\eps>0$ and $\ell$ is the length of the boundary circle); in particular boundary components are geodesic.

We consider a GFF on $\Sigma$ with covariance operator $\frac{\sigma^2}2(-\Delta)^{-1}$ and Neumann boundary condition (b.c.), and again sum over zero modes to obtain a $\sigma$-finite measure $\hat\mu$ on fields. It satisfies a Cameron-Martin formula
$$\int G(\phi+th)d\hat\mu(\phi)=\int G(\phi)\exp\left(\frac{2t}{\sigma^2}\langle \phi,h\rangle_H-\frac{t^2}{\sigma^2}\|h\|_H^2\right)d\hat\mu(\phi)$$
where $h$ is smooth on $\bar\Sigma$ with Neumann b.c. and
$$\langle\phi,h\rangle_H=\int_\Sigma(\nabla_0\phi\cdot\nabla_0 h)\omega_0=\int_\Sigma h(-\Delta_0)\phi\omega_0=\int_\Sigma \phi(-\Delta_0)h\omega_0$$
where gradient, Laplacian and area measure are defined in terms of $g_0$, and the meaning of the Green's formula on $\Sigma$ (with $\phi$ a GFF) is apparent e.g. from expanding $\phi$ and $g$ in eigenfunctions of the Laplacian with Neumann b.c.. 

We define $\TestSp$ as in Definition \ref{def:mcC}, except that $f_0,\dots,f_k$ are now smooth functions on $\overline\Sigma$ with Neumann b.c.. Then Lemma \ref{lem:GIBP} holds with these generalized definitions of $\hat\mu$, $\TestSp$ and for simplicity we take $h$ smooth with Neumann b.c.. 

We proceed to define $\nu_{g_0}$ as above, except that the GFF base measure is now understood as a Neumann GFF. Then the integration-by-parts \eqref{e:IBPsurface} holds {\em verbatim}. 

If $\hat g_0=e^{2\psi_0}g_0$ is another metric satisfying the same conditions near the boundary, note that necessarily $\psi_0$ has Neumann b.c. (see (1.14) in \cite{OPS}). The same conformal anomaly formula \eqref{e:confanom} relating $\nu_{\hat g_0}$ to $\nu_{g_0}$ holds. Once again we use as starting point the Dirichlet form (writing $\nu=\nu_{g_0}$)
$${\mc E}(F,F)=\int \|DF\|^2_{L^2(M_X)}d\nu=\int F(-{\mc L}F)d\nu(X)$$
where the expression for ${\mc L}$ is exactly as in \eqref{e:gensurf}. 

Then we define $\MeaSp={\mc M}(\overline\Sigma)\setminus\{0\}$. Remark that $\overline\Sigma$ is compact, and the class of test functions 
$$\{f\in C^\infty(\overline\Sigma): \partial_n f=0{\rm\ on\ }\partial\Sigma\}$$
used in the definition of $\TestSp$ is dense in $C_0(\overline\Sigma)$, so that Lemma \ref{lem:dense} also holds in the present set-up.

In conclusion we have:

\begin{Thm} \label{thm:surface}
Let $(\Sigma,g_0)$ be a bordered Riemannian surface. For $ \gamma<\gamma_{L^1} = 2$,
 there exists a Markov diffusion process
  $\mathbf A=\{\Omega,\mc F, (A_t)_{t\ge 0}, (P_z)_{z\in\MeaSp} \}$ on the space $\MeaSp$, such that for any smooth function $f$  and quasi-every $z\in \MeaSp$, $A_t(f)$ satisfies the  SDE \eqref{e:Ricci-Af-surf}
%  \begin{equ} [e:1d-proj-surf]
%d   A_t(f)
%=\Big( \frac{\gamma}{2\pi}  \area_0( f\Delta\phi_t )
%-\frac{Q\gamma}{2\pi}\area_0( fK_0 )
%-\mu\gamma^2    A_t ( f )\Big)dt
%+\sqrt 2 \gamma\left( A_t(f^2)\right)^{\frac12}d\beta_t^f \;,
%\end{equ}
with $A_0(f)=z(f)$,
where $\forall t>0$, $\phi_t=\mathbf M^{-1} A_t$ a.s. and
$\beta^f $ is a  one-dimensional standard Brownian motion.
 \end{Thm}

In particular for $f\equiv 1$ (taking into account Gauss-Bonnet: $\int_\Sigma K_0 \area_0=2\pi\chi$, where $\chi$ is the Euler characteristics of $\Sigma$), we see that the total volume $A_t(1)$ evolves as:
\begin{align}\label{eq:totalmassSDE}
dA_t(1)=\gamma\sqrt 2\sqrt{A_t(1)}d\beta_t-\mu\gamma^2 A_t(1)dt-Q\gamma\chi dt \;,
\end{align}
where $\beta$ is a one-dimensional Brownian motion. We will discuss long-term behavior in the more general set-up of Section \ref{Sec:insert}.

\subsection{Insertions}\label{Sec:insert}

In the context of LCFT, it is natural to consider insertions of ``vertex operators". Here we highlight the modifications needed to incorporate insertions in the SRF framework.

As before, $\Sigma$ is a (bordered) Riemann surface; additionally, we assign real weights $\alpha_1,\dots,\alpha_k$ to marked points $x_1,\dots,x_k$ in the bulk (i.e., in the interior; one could also consider boundary vertex insertions, which we refrain from for brevity). Associated to this data, we consider the formal dynamics 
(which is \eqref{e:Ricci-A-surf} with an additional term)
\begin{equation}\label{e:Ricci-A-inser}
\partial_t A_t =\frac{\gamma}{2\pi}\Lap X \area_0-\frac{Q\gamma}{2\pi}K_0 \area_0-\mu\gamma^2 A_t
+\gamma\sum_{i=1}^k\alpha_i\delta_{x_i}
+\gamma\sqrt 2\xi_g A_t
\end{equation}
and the corresponding 1d dynamics for $f$ a smooth function on $\Sigma$:
\begin{equs}[e:dA-vertex]
d A_t(f) & =\left(\frac{\gamma}{2\pi} \area_0( f\Lap X)  
 -\frac{Q\gamma}{2\pi} \area_0( fK_0)
-\mu\gamma^2 A_t( f)\right)dt
\\
& \qquad 
+\gamma\sum_{i=1}^k\alpha_if(x_i)dt
+\gamma\sqrt 2\left(A_t( f^2)\right)^{1/2}d\beta^f_t  \;.
\end{equs}

For a reference metric $g_0$ on $\Sigma$, consider the measure on fields, written for now formally as
$$d\nu_{g_0}^\alpha(X)=\left(\prod_{i=1}^k\Wick{e^{\alpha_iX(x_i)}}_{g_0}\right)d\nu_{g_0}(X)\;.$$
Note that the term in brackets is not a Radon-Nikod\'ym derivative. Admit for now the conformal anomaly formula
$$\int F(X)d\nu_{\hat g_0}^\alpha(X)\propto\int F(X-Q\psi_0)d\nu_{g_0}^\alpha(X)$$
(here $\hat g_0=e^{2\psi_0}g_0$; the coefficient of proportionality depends on the $\alpha_i$'s), and the integration-by-parts formula
\begin{equs}\label{e:IBP-vertex}
\int &D_hF(X)d\nu_{g_0}^\alpha(X) 
\\
&=\int F(X)\bigg(\frac{1}{2\pi} \langle  \nabla_0 h, \nabla_0 X\rangle_{g_0}
+\frac{Q}{2\pi} \area_0( K_0h)+\mu\gamma M_X( h) -\sum_i\alpha_ih(x_i)\bigg)d\nu_{g_0}^\alpha(X)\;.
\end{equs}
Given this, one can consider the Dirichlet form
$${\mc E}(F,F)=\int \|DF\|^2_{L^2(A)}d\nu_{g_0}^\alpha=\int F(-{\mc L}F)d\nu^\alpha_{g_0}(X)
$$
where 
\begin{equs}
{\mc L}F&=\gamma^2\sum_{i,j}\partial^2_{ij} q A( f_if_j )
\\
&+\sum_i \partial_i q \bigg( \frac\gamma{2\pi} \area_0 (f_i\Lap X)
-\frac{Q\gamma}{2\pi} \area_0 (f_i K_0)
-\mu\gamma^2 A (f_i)+\gamma\sum_j\alpha_jf_i(x_j)\bigg)
\end{equs}
which realizes the desired dynamics \eqref{e:Ricci-A-inser}.

We refer to \cite{guillarmou2016polyakov} for a construction of $\nu^\alpha_{g_0}$ with the desired properties (viz. anomaly and integration by parts); concretely, it can be realized as a vague limit of measures absolutely continuous w.r.t. $\nu_{g_0}$:
$$d\nu_{g_0}^\alpha(X)
=\lim_{\eps\searrow 0} \bigg(\prod_i\eps^{\frac{\alpha_i^2}2}e^{\alpha_iX_\eps(x_i)}\bigg)d\nu_{g_0}(X)$$
where $X_\eps$ denotes a mollification of $X$ on scale $\eps$ and $\nu_{g_0}$ is defined as in Section \ref{sec:surfaces}. This requires the following local Seiberg bound:
$$\alpha_i<Q$$
for $i=1,\dots, k$ (remark that here we are not concerned with finiteness of $\nu_{g_0}^\alpha$).  

%With the integration-by-parts formula \eqref{e:IBP-vertex} and 
Therefore, we have all the ingredients as input to the theory of Dirichlet forms.
By repeating the same arguments as in Section~\ref{sec:solution-Dirichlet}
and Section~\ref{sec:surfaces}, we have:

 \begin{Thm} \label{thm:vertex}
Let $(\Sigma,g_0)$ be a bordered Riemannian surface with marked points $x_1,\dots,x_k$ in the bulk and weights $\alpha_1,\dots,\alpha_k<Q$. For $ \gamma<\gamma_{L^1} = 2$,
 there exists a Markov diffusion process
  $\mathbf A=\{\Omega,\mc F, (A_t)_{t\ge 0}, (P_z)_{z\in\MeaSp} \}$ on the space $\MeaSp$, such that for any smooth function $f$  and quasi-every $z\in \MeaSp$, $A_t(f)$ satisfies the  SDE \eqref{e:dA-vertex}
with $A_0(f)=z(f)$,
where $\forall t>0$, $\phi_t=\mathbf M^{-1} A_t$ a.s. and
$\beta^f $ is a  one-dimensional standard Brownian motion.
 \end{Thm}

The total mass $(A_t(1))$ satisfies an SDE with generator
$$\gamma^2x\partial_{xx}-\mu\gamma^2x\partial_x+\gamma(\bar\alpha-Q\chi)\partial_x$$
where $\bar\alpha=\sum_i\alpha_i$. We now briefly explain how to analyze the long-time behavior of that process, a problem in classical one-dimensional diffusions.

The term with $\mu$ does not change qualitatively the behavior at 0, by change of measure. Setting $\mu=0$, the generator is proportional to that of a BESQ$(\delta)$, viz. $2x\partial_{xx}+\delta\partial_x$, where
$$\delta=\frac{2}\gamma(\bar\alpha-Q\chi)\;.$$
If $\delta\geq 2$, the process $(A_t(1))$ does not hit 0 (but can be started from 0). If $\delta\in (0,2)$, the process $(A_t(1))$ hits 0, but can be continued. If $\delta\leq 0$, $0$ is absorbing. 

Recall that the global Seiberg bound is $\bar\alpha- Q\chi>0$, which is precisely $\delta>0$; together with the local Seiberg bounds ($\alpha_i<Q$), it ensures finiteness of $\nu^\alpha$, see \cite{guillarmou2016polyakov}.
Since $0\notin\MeaSp$ by definition, the SRF is by construction absorbed at 0. The problem of continuation of the SRF $(A_t)$ past its first hitting time of zero dominates the corresponding problem for the one-dimensional total mass process $(A_t(1))$. However, it seems plaudible that the condition is the same, i.e. that if $\delta\in (0,2)$, the SRF can be extended to a process on $\MeaSp\cup\{0\}$, with infinite lifetime (i.e. conservative), see Questions 5-6 below.

\section{Questions and open problems}\label{sec:open-probs}

\paragraph{Regularity.} In the 2d Stochastic Heat Equation $\partial_t\phi_t=\Delta\phi_t+\xi$, the solution can be realized as an element of $C([0,T],H^{-s})$ for any $s>0$, i.e. $t\mapsto\phi_t$ is a.s. continuous w.r.t. a Banach space topology.

For a Stochastic Ricci Flow $(\phi_t,A_t)$, we know that the second marginal $t\mapsto A_t$ evolves a.s. continuously w.r.t. to the weak topology on $\MeaSp$. 

{\bf Question 1.} Strenghten the regularity of the second marginal, e.g. show that $t\mapsto A_t$ is continuous a.s. in a Besov space topology. 
Note that \cite{MR3531686} has proved a H\"older continuity result for the GMC $M_\gamma$ with $\gamma\in [0,2)$ on the torus: for any $\eps>0$, almost surely there is a (random) constant $C$ depending on $\eps$ and the Gaussian free field, such that $M_\gamma(B(x,r))\le C r^{\alpha-\eps}$ with $\alpha=2(1-\frac{\gamma}{2})^2$ for any ball $B(x,r)$.
We also note that the space-time regularity  (with respect to a parabolic distance) of the GMC $\M$ when the field $\phi$ evolves according to stochastic heat equation  is obtained  in \cite{Garban2018}.

{\bf Question 2.} Is the first marginal a.s. continuous w.r.t. an abstract Wiener space topology ? (e.g., $H^{-s}$ for some $s$)

\paragraph{Feller property.}

The Dirichlet form formalism provides a family of probability measures on path space $(P_z)_{z\in\MeaSp}$ indexed by starting state $z\in\MeaSp$;
this family is uniquely defined except possibly on an exceptional set of starting states (see Theorem 4.2.7 in \cite{Fukushima2011}). This exceptional set is in particular $\m$-negligible.

{\bf Question 3.} Define $P_z$ unambiguously for all $z\in\MeaSp$.

This can be thought of as an entrance problem. Concretely, if $(U_n)$ is a sequence of shrinking neighborhoods of $z\in\MeaSp$, one would expect $(P_{\m(\cdot|U_n)})_n$ to converge weakly as $n\rightarrow\infty$ (w.r.t. to the Skorohod space topology on $C([0,\infty),\MeaSp\cup\{0\})$).

{\bf Question 4.} Is the Stochastic Ricci Flow strong Feller ? (w.r.t. to a topology on $A$ or $(\phi,A)$)

Here it may be useful to have more explicit control on the map ${\mathbf M}^{-1}$ (see \eqref{e:mapM}).

\paragraph{Entrance and reflection.}

It follows from strong locality (proof of Proposition \ref{prop:existsA}) that, on the event that the lifetime $\zeta$ is finite, the total mass $(A_t(1))$ goes to $0$ or $\infty$ as $t\nearrow\zeta$. The latter is ruled out by the autonomous SDE satisfied by the mass (see \eqref{eq:totalmassSDE}), and the former happens a.s. iff $\delta=\frac 2\gamma(\bar\alpha-Q\chi)<2$ (see the discussion at the end of Section \ref{Sec:insert}). By construction (i.e., the choice of state space and core), the process is absorbed at 0 (and hence not conservative if it hits zero in finite time). The symmetrizing measure $\m$ is invariant for the semigroup if the semigroup is conservative. By comparison of the total mass process with BESQ$(\delta)$, it is natural to ask:

{\bf Question 5.} In the case $\delta\geq 2$, is 0 an entrance boundary ? viz., can the SRF be extended to a Feller process on $\MeaSp\cup\{0\}$, that a.s. never returns to 0 ?

{\bf Question 6.} In the case $\delta\in (0,2)$, can the SRF be extended to a process reflected at 0 ? i.e. a conservative process such $\{t:A_t(1)=0\}$ has a.s. zero Lebesgue measure.

If one starts with $\bar\MeaSp:=\MeaSp\cup\{0\}$ as state space, the difficulty is to define a core ${\mc C}_{\bar\MeaSp}$ dense in $C_0(\bar\MeaSp)$, such that integration by parts \eqref{e:DDform} still holds. 

\paragraph{Approximation schemes.}

Fix two small parameters $\delta,\eps$. Based on \eqref{eq:smoothSHE}, it is natural to consider the following scheme: on the time interval $t=[k\delta,(k+1)\delta)$, solve the linear SHE with smooth coefficients (here in the torus case)
$$\partial_t\phi_t=\eps^{\alpha}e^{-2\phi^\eps_{k\delta}}\Delta_0\phi_t+\sigma \eps^\beta e^{-\phi^\eps_{k\delta}}\xi_0$$
Here $\phi^\eps$ denotes an $\eps$-mollification of $\phi$. Remark that, if $\phi_0$ is absolutely continuous w.r.t. to the GFF \eqref{eq:pathint}, then for all $t$, $\phi_t$ is also absolutely continuous w.r.t. the same GFF. This is a natural analogue of frozen coefficients approximations for SDEs.

{\bf Question 7.} Does this scheme converge to SRF as $(\delta,\eps)\rightarrow (0,0)$ in some way, for suitable renormalization exponents $(\alpha,\beta)$ ?

Showing directly the convergence of such a scheme could provide an alternative proof of existence and shed some light on the previous regularity questions. 

{\bf Question 8.} Find a Wong-Zakai approximation for small $\sigma$.

Here one considers $\xi^\eps$, a space-time $\eps$-mollification of the white noise $\xi$; then one solves classically
$$\partial_t\phi^\eps=\eps^{\alpha}e^{-2\phi^\eps}\Delta\phi -\lambda + \eps^{\beta} \sigma e^{-\phi^\eps} \xi^\eps+({\rm counterterms})$$
and attempts to take a limit in probability as $\eps\searrow 0$, for suitable normalization exponents $(\alpha,\beta)$ and - possibly - counterterms.

Another direction would be to consider scaling limits of natural dynamics on discretizations of LQG, such as random maps. In that case, dynamics such as triangle flipping in random triangulations have been proposed. However a feature of the SRF is that the complex structure is preserved, and it is unclear which discrete dynamics can be expected to have this property, even in the limit.

\paragraph{Strong solutions.}
It is not immediately apparent how to phrase a notion of strong solutions for SRF. The previous approximation schemes (for fixed $\eps>0$) are measurable with respect to a fixed white noise $\xi_0$.

{\bf Question 9.} Show almost sure convergence of an approximation scheme, for a fixed realization of $\xi_0$.

Formally, the SRF in terms of $\phi$ \eqref{e:SRicci} is a quasilinear singular stochastic PDE. Strong solution theories for such quasilinear  stochastic PDEs are under rapid progress. In \cite{otto2018quasilinear}, strong solutions to 
equations of the form (up to technical subtleties such as a mean-zero component projection therein) 
$\partial_t u = a(u) \partial_x^2 u + \sigma(u) f$ for random forcing $f\in C^{\alpha-2}$ with $\alpha>\frac23$ are constructed using controlled rough paths theory. 
Here $C^\alpha$ is a space-time H\"older regularity defined with respect to a parabolic distance, see  \cite[Section~2]{otto2018quasilinear}.
The solution lies in the space $C^\alpha$ and  is the limit of a sequence of suitably renormalized equations driven by smooth mollified noises (that is, $f$ convolved with smooth mollifiers).
Note that $\alpha=1$ is the borderline where  the products $a(u) \cdot \partial_x^2 u$ and  $\sigma(u) \cdot f$  fail to have a classical meaning.

The key idea that allows \cite{otto2018quasilinear} to generalize the strong solution theories such as \cite{Regularity} that was originally applied to study semilinear equations is a parametric ansatz; one builds solution to a  family of linear equations $\partial_t v  = a_0 \partial_x^2 v + f$ parametrized by constants $a_0$, as well as higher order terms $vf$ and $v\partial_x^2 v$.
The input $(v,vf,v\partial_x^2 v)$, once constructed by stochastic methods, is sufficient to render a PDE theory as long as $\alpha>\frac23$,
because the ``error'' of replacing $a(u)$ or $\sigma(u)$ by $v$ is order $2\alpha$ and $2\alpha + (\alpha-2)>0$ is the key condition for PDE estimates.
Similar results have been obtained by 
\cite{MR3916943,MR3916262} also for $\alpha>\frac23$, but using the para-controlled approach (originally developed in \cite{MR3406823}).

The work by \cite{GerencserQuasilinear}  then generalized the above results by building a framework for 
 construction of local renormalized solutions to general quasilinear stochastic PDEs within the theory of regularity structures. It exploited  a series  of existing results developed for the semilinear case such as \cite{bruned2016algebraic,bruned2017renormalising,chandra2016analytic} so that it only requires a small number of additional arguments to extend to the quasilinear setting.
 As applications an equation of the form
 $\partial_t u = a(u) \partial_x^2 u + F(u)(\partial_x u)^2+ \sigma(u) f$ is considered where $f\in C^{\alpha-2}$ with $\alpha >\frac12$. There is also \cite{OttoSauerSmithWeber} under a twisted version of regularity structure framework which works for $\alpha >\frac12$.
With extra work, one may expect to push the regularity down to $\alpha>\frac25$ by building more ``perturbative'' information so that $4\alpha + (\alpha-2)>0$.
But this would eventually cease to work at $\alpha=0$
and the SRF should be as singular as the two-dimensional GFF, i.e. $\alpha<0$. Note that spatial dimension is two for SRF here, but the obstacle here is regularity rather than dimension (some of the aforementioned papers work or can be adapted to more than one dimension.)

\vspace{1ex}

Alternatively, one can attempt to construct a solution as a small noise expansion (e.g. in powers of $\sigma$); the terms in the expansion are then measurable with respect to a standard white noise $\xi_0$.

Note that Takhtajan \cite{MR2249798} defines Liouville CFT via perturbative expansion (in Planck constant) around the classical solution $\phi_{cl}$ to the Liouville equation. It is not clear to us how to ``translate'' his works to the dynamic setting.  However, here are some thoughts. Consider again \eqref{e:SRicci}.
We can try to write the solution $\phi$ as a series in $\sigma$:
\[
\phi= \sum_{i=0}^\infty \sigma^i \phi_i \;.
\]
We can show that  each $\phi_i $ satisfies an equation that is {\it linear} in $\phi_i$, and only depending on the $\phi_j$ ($j<i$). In particular,
\begin{equ}
\partial_t\phi_0 
=e^{-2\phi_0}\Delta\phi_0 -\lambda
\end{equ}
which is the classical Ricci flow, and a stationary solution (set $\partial_t\phi_0 =0$) is the  solution to the classical Liouville equation. Also,
\begin{equs}
\partial_t\phi_1 
&=e^{-2\phi_0}\Delta\phi_1 - 2 e^{-2\phi_0} \phi_1 \Lap \phi_0 +e^{-\phi_0} \xi_0
\\
%&=e^{-2\phi_0}\Delta\phi_1 +e^{-\phi_0} \xi 
\partial_t\phi_2 
&= e^{-2\phi_0}\Delta\phi_2
+2 e^{-2\phi_0} (\phi_1^2 - \phi_2) \Lap \phi_0 - 2 e^{-2\phi_0} \phi_1 \Lap\phi_1 
- e^{-\phi_0} \phi_1  \xi \;.
\end{equs}
Note  that  the second order operator in each of the equations is
$\partial_t - e^{-2\phi_0}\Delta $.
This seems  close to the spirit of Takhtajan  \cite{MR2249798}  who takes $(e^{-2\phi_{cl}}\Delta+m^2)^{-1}$ as the ``free propagator''
in his Feynman diagram expansion.

{\bf Question 10.} Can one define a series solution via small noise expansions?

\bibliographystyle{alpha}
\bibliography{refs}

\end{document}